\documentclass[twoside,11pt]{amsart}

\usepackage{graphicx,amsfonts,amssymb,amsmath,amsthm,url,verbatim, mathtools}
\usepackage{todonotes}
\usepackage{a4wide}
\usepackage{booktabs} 
\usepackage[all, cmtip]{xy} 
\usepackage{tikz-cd} 
\usepackage{pdflscape}
\usepackage{mathabx}
\def\acts{\lefttorightarrow}

 \newtheorem{thm}{Theorem}
 \newtheorem{prop}[thm]{Proposition}
 \newtheorem{cor}[thm]{Corollary}
 \newtheorem{lem}[thm]{Lemma}
  
 \newtheorem{conj}[thm]{Conjecture}
 \newtheorem*{bthm}{Theorem}

\theoremstyle{definition}
\newtheorem*{defn}{Definition}

\theoremstyle{remark}
\newtheorem{rem}[thm]{Remark}

\allowdisplaybreaks	

\newcommand{\Z}{\mathbb{Z}}
\newcommand{\Q}{\mathbb{Q}}

\newcommand{\C}{\mathbb{C}}
\newcommand{\F}{\mathbb{F}}
\renewcommand{\Im}{\mathrm{Im}}
\newcommand{\cwe}{\mathrm{cwe}}

\newcommand{\Tr}{\mathrm{Tr}}
\newcommand{\GL}{\mathrm{GL}}

\newcommand{\Sp}{\mathrm{Sp}}

\renewcommand{\(}{\left(} \renewcommand{\)}{\right)}

\newcommand{\T}[1]{\,{}^t\! {{#1}}} 

\begin{document}
\title{On an analogue of the doubling method \linebreak in coding theory}

\author{Thanasis Bouganis and Jolanta Marzec-Ballesteros}
\address{Department of Mathematical Sciences, Durham University,
 Durham, UK.}
 \email{athanasios.bouganis@durham.ac.uk}
\address{Faculty of Mathematics and Computer Science, Adam Mickiewicz University, Pozna\'n, Poland}
\email{jmarzec@amu.edu.pl}

\begin{abstract}
It is well known that there is a deep relationship between codes and lattices. Concepts from coding theory are related to concepts of lattice theory as, for example, weight enumerators to theta series, MacWilliams identity to Jacobi identity, and Gleason's theorem to Hecke's theorem. In this framework, higher-genus (or multiple) weight enumerators are related to Siegel theta series, which opens up the possibility of introducing concepts from the theory of higher-rank modular forms to coding theory. There has been important work in this direction, for example Runge introduced a coding theory analogue of Siegel's $\Phi$-operator and Nebe analogues of Hecke operators. In this paper, we show that the celebrated Doubling Method from the theory of higher-rank modular forms has a coding theory analogue. Given the impact that the Doubling Method has had in the study of higher-rank modular forms, one may expect that its analogue may prove useful to the study of higher-genus weight enumerators. In this paper we use it to solve an analogue of the ``basis problem''. That is, we express ``cuspidal'' polynomials which are invariant  under a Clifford-Weil type group as an explicit linear combination of higher-genus weight enumerators of self-dual codes of that type.  
\end{abstract}

\maketitle

\section{Introduction}
It is now a well established fact that there is a deep connection between the theory of lattices and self-dual codes. Indeed, many concepts and theorems in one theory have analogues in the other. For example, an analogue of the celebrated theorem of Hecke on the generators of the ring of modular forms is the theorem of Gleason in the theory of self-dual codes and their weight distribution. 
We refer to the preface of the book \cite{NRS06}, and especially the table on page vii, for a list of analogies between the two theories. As it is indicated in that table, the analogies between the two theories go beyond an analogy between classical (or elliptic) modular forms and the (Hamming-)weight distribution of self-dual codes. In particular, using the concept of higher-order weight enumerators one may extend the analogy to include higher rank modular forms such as Siegel modular forms. The main aim of this paper is to find an analogue of the doubling method in the theory of error-correcting codes. Along the way we will emphasize some of the analogies alluded to above.

The doubling method has been a powerful tool in the theory of higher-rank modular forms (or even more generally, automorphic forms) leading to many important applications in the study of $L$-values, decomposition of Eisenstein series and the basis problem, to name a few. The method was initiated in the 1980's by Garrett (in the classical setting) and by Piatetski-Shapiro and Rallis (in the setting of automorphic representations), and since then has been vastly generalised. We will not attempt to give a detailed exposition of the doubling method in this paper, but we do include a short appendix which may serve as a very basic introduction. The appendix is mainly intended for readers of coding theoretic background.

In this article we are only interested in self-dual linear codes over a finite field $\F_q$. For more precise information and for explanations of the notions used below, we refer the reader to section \ref{background}. To a code $C$, for any integer $g\geq 1$, one can associate a homogeneous polynomial of $q^g$ variables - a genus $g$ weight enumerator $\cwe_g(C, (x_v)_{v\in\F_q^g})$. It is known that, for a fixed $g$ and for the codes of a fixed type, the $\C$-algebra generated by these polynomials is invariant under the action of a Clifford-Weil group $\mathcal{C}_g$ (whose definition depends on the type of the codes). 
This algebra may be thought of as an analogue of the space of Siegel modular forms of genus $g$ (of suitable weights and level). Based on this analogy we may speak of Eisenstein series and cusp forms. Of special interest to us is the following homogeneous polynomial, which can be thought of as the coding theory analogue of an Eisenstein series of Siegel type,
$$E_g((x_v)_{v\in\F_q^g})=\frac{1}{|\mathcal{P}_g\backslash\mathcal{C}_g|}\sum_{\sigma\in \mathcal{P}_g\backslash \mathcal{C}_g} \left(\sum_{v\in\F_q^g} x_v^N\right)^\sigma.$$ 
In this paper, we will calculate the inner product of such an ``Eisenstein series'' (restricted to a domain of the form $\mathcal{F}\times\mathcal{F}$) and a function (on a domain $\mathcal{F}$), which is an eigenfunction of certain averaging operators, and show that it is equal to a constant multiple of that function. In the theory of modular forms, the resulting constant is of great arithmetic interest as it is related to the aforementioned eigenvalues. Our main result shows that in the context of coding theory, the role of these functions is played by cusp forms whose genus is half the genus of the Eisenstein series. We will prove the following. 

\begin{bthm}
Let $f\in\mathbb{C}[y_v: v\in\F_q^g]$ be a homogeneous polynomial of degree $N$ which is invariant under the action of a Clifford-Weil group $\mathcal{C}_g$. If $\mathcal{C}_g$ is associated with codes of type $2^E_{II}$, $q_1^E$ or $q^E$ and $f$ is a cusp form of genus $g$, then
\[
\left( D(E_{2g})(\underline{x},\underline{y}),f(\underline{y})\right)_g = cN! \bar{f}(\underline{x}),
\]
where the constant 
\[
c=\begin{cases}
2^{g^2+2g-Ng/2}\prod_{i=1}^g\frac{2^i-1}{2^{g+i}+1} , & \mbox{for type } 2^E_{II}\\
q^{g^2+2g-Ng/2}\prod_{i=1}^g\frac{q^i-1}{q^{g+i}+1} , & \mbox{for type } q^E_{1}\\
q^{g^2-Ng/2}\prod_{i=1}^g\frac{q^i-1}{q^{g+i}+1} , & \mbox{for type } q^E.
\end{cases}
\]
is independent of $f$. Here $\bar{f}(\underline{x})$ denotes the polynomial obtained from $f(\underline{x})$ by applying complex conjugation to the coefficients. 
\end{bthm}

In the above theorem, $D(E_{2g})$ denotes the aforementioned restriction of the Eisenstein series $E_{2g}$ from the space $\F^{2g}_q$ to the space $\F^{g}_q\times\F^{g}_q$, and $\underline{x},\underline{y}\in\F_q^g$ (see the beginning of section \ref{sec:doubling-method} for the definition of this map); the inner product $\left(\; ,\,\right)_g$ is defined in section \ref{sec:cusp forms}.
At the end of section \ref{sec:r=g} we formulate a conjecture regarding a similar statement for the codes of type $2^E_I$; it is verified in a couple of cases in section \ref{sec:computation}. The proof of the theorem is carried out throughout sections \ref{sec:double-coset-decomp}-\ref{sec:r=g}, following the strategy described in the appendix.

As we mentioned above, the constant $c$ appearing in the theorem should be related to eigenvalues of $f$ with respect to the action of certain averaging operators. The candidates for such operators could be Hecke operators introduced in \cite{N09} or Kneser-Hecke operators described in \cite{N06}. In fact, as shown in \cite[Corollary 19]{N09}, the eigenvalues of these operators are related to each other. As it is the case for our constant $c$, these eigenvalues depend on $\frac{N}{2}$, the genus $g$ and are constant on the spaces of cusp forms of a given genus. It would be interesting to see if they are related to the constant we obtain here.

In the case of codes of type $2^E_{II}$ and type $p_1^E$, with $p$ a prime number, it is possible (see \cite{NRS01}) to express the Eisenstein series $E_{2g}$ as a sum of genus-$2g$ complete weight enumerators of all the codes of fixed type and length. From this and our main theorem it immediately follows that any cusp form $f$ of genus $g$ has a Fourier-type expansion in terms of genus-$g$ complete weight enumerators of all the codes of the same type and length:
\[
f = \frac{b}{cN!} \sum_C \left(f, \cwe_g(C)\right)_g \cwe_g(C);
\]
the constant $b$ is given in the Corollary in section \ref{sec:basis problem}.

We close this introductory section by mentioning that in the theory of higher rank modular forms the doubling method can be extended to the case where the Siegel-type Eisenstein series is restricted to a product of two different domains (in the notation above $\mathcal{F}_1 \times \mathcal{F}_2$). For example, in the theory of Siegel modular forms this corresponds to an embedding of the form $\Sp_{n} \times \Sp_m \hookrightarrow \Sp_{n+m}$; the case discussed in this paper is for $n=m$. In this more general setting, the doubling method furnishes Klingen-type Eisenstein series. These do have an analogue in the coding theory (liftings of cusp forms) and we believe that our techniques here can be generalized to include also this analogy. We defer this to a future work.

\section{Background and main definitions}\label{background}
This article concerns linear codes over a finite field $\F$, that is, vector spaces $C\subset\F^N$. In characteristic $2$ we assume that $\F=\F_2$, so that the codes are binary. We equip $\F^N$ with the bilinear form $\beta^{(N)}(x ,y) = \sum_{i=1}^N \beta (x_i, y_i)$, where $\beta :\F\times\F\to \Q/\Z$ is specified below, and assume that $C$ is self-dual with respect to $\beta^{(N)}$, that is, $C$ is equal to its dual code
$$C^{\perp}=\{ c'\in\F^N:\forall_{c\in C}\, \beta^{(N)} (c',c)=0\} .$$
Then $N$ is an even natural number and the code $C$ is $N/2$-dimensional. According to the notation of \cite{NRS06}, we consider the codes of the following types:
\begin{itemize}
\item $2_I^E$: self-dual binary codes equipped with a bilinear map $\beta :\F_2\times\F_2\to\frac12\Z/\Z$, $\beta (x,y)=\frac12 xy$ and quadratic forms $\Phi=\{\varphi :x\mapsto\frac12 x^2,0\}$;
\item $2_{II}^E$: doubly-even self-dual binary codes equipped with a bilinear map $\beta$ as above and quadratic forms $\Phi=\{\phi :x\mapsto\frac14 x^2,2\phi ,3\phi ,0\}$; a binary code $C$ is doubly-even if for every codeword $c=(c_1,\ldots , c_N)\in C$, the weight $|\{ i:c_i\neq 0\}|$ is divisible by $4$; 
\item $q^E$: self-dual codes over a field $\F_q$ with $q=p^f$ odd, equipped with a bilinear map\\ $\beta :\F_q\times\F_q\to\frac1p\Z/\Z$, $\beta(x,y)=\frac1p \Tr(xy)$ and quadratic forms $\Phi=\{\varphi_a :x\mapsto\frac1p\Tr(ax^2)| a\in\F_q\}$; 
\item $q_1^E =\{ C\in q^E: (1,\ldots ,1)\in C\}$: equipped with  $\beta :\F_q\times\F_q\to\frac1p\Z/\Z$, $\beta(x,y)=\frac1p \Tr(xy)$ and $\Phi=\{\varphi_{a,b} :x\mapsto\frac1p\Tr(ax^2+bx)| a,b\in\F_q\}$;
\end{itemize}
where $\Tr$ denotes the trace of $\F_q$ over the field $\F_p$.

\subsection{Complete weight enumerators and the associated  Clifford-Weil groups}\label{sec:Clifford}
The genus-$g$ complete weight enumerator of a code $C$ of length $N$ is a homogeneous polynomial in $\C[x_v: v\in \F^g]$ of degree $N$:
$$\cwe_g(C)=\sum_{\underline{c}\in C^g}\prod_{v\in\F^g} x_v^{a_v(\underline{c})} ,$$
where $a_v(\underline{c})$ is the number of occurrences of $v$ as a row in the $N\times g$ matrix of column vectors defined by $\underline{c}$. Runge \cite{Runge96} (for the types $2_I^E$ and $2_{II}^E$) and Nebe, Rains, Sloane \cite[Corollary 5.7.6]{NRS06} (in much greater generality) proved that as $C$ varies over codes of one of the types above, the weight enumerators $\cwe_g(C)$ span the space of invariants of degree $N$ in $\C[x_v: v\in \F^g]$ under the action of the associated Clifford-Weil group $\mathcal{C}_g$. The group $\mathcal{C}_g$ 
is a subgroup of $\GL_{|\F|^g}(\C)$ generated by the following elements:
$$m_u: x_v\mapsto x_{uv},\qquad u\in\GL_g(\F)$$
\begin{equation}\label{generators}
d_\phi : x_v\mapsto \exp(2\pi i \phi(v))x_v,\qquad \phi\in\Phi^{(g)}
\end{equation}
$$h_{\iota,u_\iota,v_\iota}: x_v\mapsto \frac{1}{\sqrt{|\iota\F^g|}}\sum_{w\in \iota\F^g}\exp( 2\pi i\beta^{(g)}(w,v_\iota v)) x_{(1-\iota)v+w} ,$$
where $\exp t=e^t$, $\iota=u_\iota v_\iota$ varies over symmetric idempotents\footnote{We say that $\iota\in M_g(\F)$ is a symmetric idempotent if there exists an isomorphism $\kappa: \iota  M_g(\F)\to\T{\iota} M_g(\F)$. Any such $\kappa$ is defined as $\kappa (\iota x)=v_\iota x$, $\kappa^{-1}(\T{\iota}x)=u_\iota x$, where $u_\iota\in \iota M_g(\F)\T{\iota}$ and $v_\iota \in\T{\iota}  M_g(\F)\iota$ satisfy $u_\iota v_\iota=\iota$, $v_\iota u_\iota=\T{\iota}$.} in $M_g(\F)$, and the superscript $(g)$ means that $\phi\in\Phi^{(g)}$ is a form on $\F^g$ and 
$$\Phi^{(g)} =\left\{\begin{pmatrix} \phi_1 & m_{12} & \ldots & m_{1g}\\
 & \ddots & \ddots & \vdots\\ 
  & & \ddots & m_{g-1,g}\\ & & & \phi_g
\end{pmatrix}: \phi_i\in\Phi, m_{ij}\in \F\right\}$$
with 
$$\begin{pmatrix} \phi_1 & m_{12} & \ldots & m_{1g}\\
 & \ddots & \ddots & \vdots\\ 
  & & \ddots & m_{g-1,g}\\ & & & \phi_g
\end{pmatrix} (v_1,\ldots ,v_g)=\sum_{i=1}^g \phi_i(v_i) +\sum_{1\leq i<j\leq g} \beta( v_i,m_{ij}v_j) .$$
Note that $m_u$ is given by a permutation matrix, and $d_\phi$ by a diagonal matrix whose entries are $\pm 1$ (type $2^E_I$ codes), $\pm 1$ or $\pm i$ (type $2^E_{II}$ codes), $p$-th roots of unity (types $q^E$ and $q_1^E$).

As it is shown in \cite[Theorem 5.3.2]{NRS06}, the Clifford-Weil group $\mathcal{C}_g$ is the projective image of a hyperbolic co-unitary group (in the terminology and notation used there) $\mathfrak{U}(M_g(\F),\Phi^{(g)})$ under a projective representation $\pi : \mathfrak{U}(M_{g}(\F),\Phi^{(g)}) \rightarrow PGL_{|\F|^g} (\mathbb{C})$, that is 
$$\mathcal{C}_g \cong Z. \mathfrak{U}(M_g(\F), \Phi^{(g)}),$$
where
\begin{equation}\label{eq:iso of U}
\mathfrak{U}(M_g(\F),\Phi^{(g)})\cong (\ker\lambda^{(g)}\oplus\ker\lambda^{(g)} ).\mathcal{G}_g.
\end{equation}
We write $A.B$ to denote a group whose normal subgroup is isomorphic to $A$ and whose quotient is isomorphic to $B$. In the case of odd characteristic the operation $A.B$ is actually $A \rtimes B$, the semidirect product with the normal subgroup $A$ (see \cite[page 150-151]{Weil1}). The group $Z$ is a suitable centre of $\mathcal{C}_g$ which is independent of the genus $g$ (see \cite[Theorems 7.4.1, 7.6.1, 7.6.3]{NRS06} or Table \ref{table:parameters} below),
$$\lambda^{(g)} :\quad \phi\in\Phi^{(g)}\quad\mapsto\quad \left( (x,y)\;\mapsto\; \phi(x+y)-\phi(x)-\phi(y)\right) ,$$
and $\mathcal{G}_g$ is a classical group as described in Table \ref{table:parameters}. Note here that the kernel of $\lambda^{(g)}$ consists precisely of those forms in $\Phi^{(g)}$, which are additive. In particular in the case $q^E$ we have that $\lambda^{(g)}$ is injective and hence $\ker\lambda^{(g)} = \{ 0\}$, whereas in the cases $q_1^E$, $2_I^E$ and $2_{II}^E$ we have that $\ker\lambda^{(g)} \cong \F^g$ (see \cite[Theorems 7.4.1,  7.6.1, 7.6.3]{NRS06}). 
\begin{table}[ht]
\centering
\begin{tabular}{llcccc}\toprule
type & : & $2_{I}^E$ & $2_{II}^E$ & $q^E$ & $q_1^E$\vspace{0.2cm}\\
$\mathcal{G}_g$ & : & $O_{2g}^+(\F_2)$ & $\Sp_{2g}(\F_2)$ & $\Sp_{2g}(\F_q)$  & $\Sp_{2g}(\F_q)$ \vspace{0.2cm}\\
$Z$ & : & $\Z/2\Z$ & $\Z/8\Z$ & $\Z/\gcd(q+1,4)\Z$ & $\F_p \times \Z/\gcd(q+1,4)\Z$ \vspace{0.2cm}\\
$\ker\lambda^{(g)}$ & : & $\F^g$ & $\F^g$ & $\{ 0\}$ & $\F^g$\\
\bottomrule\\
\end{tabular}
\caption{The data describing Clifford-Weil groups associated with codes of type $2_{I}^E$, $2_{II}^E$, $q^E$ and $q_1^E$.}\label{table:parameters}
\end{table}

\subsection{Cusp forms}\label{sec:cusp forms}
As it is explained in \cite[Chapter 9]{NRS06} and in particular in Theorem 9.1.14 there, the genus-$g$ complete weight enumerators of a linear code $C$ are related to the theta series attached to a lattice $\Lambda_C$ related to the code $C$ by the ``Construction A''. More generally, polynomial invariants of the Clifford-Weil group may be thought of as analogues of modular forms of genus $g$ for some appropriate theta group. The analogues of the notion of a cusp form may be defined via an analogue of Siegel's $\Phi$-operator. We introduce them following an exposition given in \cite{N09}. Later, in section \ref{sec:r<g}, we present a more general characterisation.

The finite Siegel $\Phi$-operators are linear maps and ring homomorphisms
$$\Phi_{g,j}:\C[x_v: v\in \F^g]\to \C[x_v: v\in \F^{g-j}],\qquad j\in\{ 0,\ldots ,g\},$$
such that
$$x_{(v_1,\ldots ,v_g)}\mapsto \begin{cases} 
x_{(v_1,\ldots ,v_{g-j})} & \mbox{if } v_{g-j+1}=\ldots =v_g=0\\
0 & \mbox{otherwise}
\end{cases} ;$$
if $j=g$, we write $x_\circ$ for the generator of the codomain of $\Phi_{g,g}$.

Together with $\Phi_{g,j}$ we consider lifts, also ring homomorphisms,
$$\varphi_{g,j}:\C[x_v: v\in \F^{g-j}]\to \C[x_v: v\in \F^{g}],\qquad x_{(v_1,\ldots ,v_{g-j})}\mapsto x_{(v_1,\ldots ,v_{g-j},0,\ldots ,0)} ,$$
$j\in\{ 0,\ldots ,g\}$, so that $\Phi_{g,j}\circ \varphi_{g,j}$ is the identity map. On the subspace $\C[x_v: v\in \F^g]_N$ of the space $\C[x_v: v\in \F^g]$ consisting of homogeneous polynomials of degree $N$ we define a hermitian inner product $\left(\;\, ,\; \right)_g$ such that on monomials
\[
\left( \prod_{v\in \F^g} x_v^{n_v}, \prod_{v\in \F^g} x_v^{m_v}\right)_g =\begin{cases}
\prod_{v\in \F^g} (n_v!),& n_v=m_v\mbox{ for all } v\in \F^g\\
0,& \mbox{otherwise.}
\end{cases}
\]

As it is explained in \cite{N09}, this inner product may also be written as
\begin{equation}\label{def:inner product}
(p,q)_g = p\left(\left(\frac{\partial}{\partial x_v}\right)_v\right) (\bar{q})  
\end{equation}
where $p, q\in \C[x_v: v\in \F^g]_N$, and $p\left(\left(\frac{\partial}{\partial x_v}\right)_v\right)$ denotes the differential operator
obtained from the polynomial $p$ by replacing each variable $x_v$ by $\frac{\partial}{\partial x_v}$ and $\bar{q}$ is the polynomial obtained from $q$ by applying complex conjugation to its coefficients. Observe that $(p,q)_g$ defined via \eqref{def:inner product} is always a constant polynomial because $p$ and $q$ are homogeneous of the same degree. For such polynomials $p$ and $q$, let $(p\gamma )((x_v)_v):=p((x_v)_v\gamma)$, $(q\gamma )((x_v)_v):=q((x_v)_v\gamma)$. We claim that 
 \begin{equation*}
 (p\gamma,q\gamma)_g= (p,q)_g \qquad \mbox{for all}\quad \gamma \in \mathcal{C}_g .
 \end{equation*}

To see this, observe that for any $\gamma \in GL_{|\F^g|}(\C)$,
\[
(p \gamma, q \gamma)_g = p\left(\left(\frac{\partial}{\partial x_v}\right)_v\gamma \right) \left(\bar{q}( (x_v)_v \bar{\gamma})\right) .
\]
In particular, since $\gamma \in \mathcal{C}_g$ is a unitary matrix, we have
\[
(p \gamma, q \gamma)_g = p\left(\left(\frac{\partial}{\partial x_v} \right)_v\T{\bar{\gamma}}^{-1} \right) (\bar{q} ((x_v)_v \bar{\gamma})) .
\]
Further, we see from \cite[Lemma 5.6.7]{NRS06} that
\[
p\left(\left(\frac{\partial}{\partial x_v}\right)_v \T{\bar{\gamma}}^{-1} \right) (\bar{q}((x_v)_v \bar{\gamma})) =\( p\left(\left(\frac{\partial}{\partial x_v}\right)_v\right) (\bar{q})\)((x_v)_v \bar{\gamma}) = (p,q)_g ((x_v)_v \bar{\gamma}) = (p,q)_g.
\]

Hence, we can conclude that the hermitian form $(p,q)_g$ is invariant under the action of $\mathcal{C}_g$, and hence the adjoint of an element in the Clifford-Weil group with respect to this inner product is nothing else than its inverse.\newline

Moreover by \cite{N09} we have the following:
\begin{enumerate}
\item For homogeneous polynomials $p\in \C[x_v: v\in \F^{g}]_N$ and $q\in \C[x_v: v\in \F^{g-j}]_N$ it holds that 
$$(\varphi_{g,j}(q),p)_g=(q,\Phi_{g,j}(p))_{g-j}.$$
\item $\varphi_{g,j}\circ \Phi_{g,j}$ is a self-adjoint idempotent in the space of endomorphisms of $\C[x_v: v\in \F^{g}]_N$.
\item The image of $\varphi_{g,j}$ is the orthogonal complement of the kernel of $\Phi_{g,j}$ 
and we have an orthogonal decomposition
\begin{align*}
\C &[x_v: v\in \F^{g}]_N =\ker(\Phi_{g,1})\perp\varphi_{g,1}(\C[x_v: v\in \F^{g-1}]_N)\\
&= \ker{\Phi_{g,1}}\perp\varphi_{g,1}(\ker (\Phi_{g-1,1}))\perp\varphi_{g,2}(\ker (\Phi_{g-2,1}))\perp\ldots\perp \varphi_{g,g-1}(\ker (\Phi_{1,1}))\perp\varphi_{g,g}(\C [x_\circ]_N )
\end{align*}
\end{enumerate}

\begin{defn}
Fix an integer number $g\geq 1$. We say that a non-zero $\C$-linear combination $f$ of genus-$g$ complete weight enumerators $\cwe_g(C)$ of some self-dual codes $C$ of fixed length and type is a cusp form of genus $g$ if $f\in \ker(\Phi_{g,1})$. More generally, we say that a non-zero homogeneous polynomial $f\in \C[x_v: v\in \F^{g}]$ is a cusp form of genus $g$ for a Clifford-Weil group $\mathcal{C}_g$ if $f$ is invariant under the action of $\mathcal{C}_g$ and $f\in \ker(\Phi_{g,1})$.
\end{defn}

For the codes of types $2_{I}^E$ and $2_{II}^E$ the dimensions of the spaces of cusp forms have been computed for $N\leq 32$ (see \cite[Tables 1 and 2]{N06}).

\subsection{Eisenstein series}
Fix genus $g\geq 1$ and a type of a code. The type determines the Clifford-Weil group $\mathcal{C}_g$ and then the length $N$: $N$ must be divisible by $|Z|$ (the size of the centre of $\mathcal{C}_g$). Let $\mathcal{P}_g$ be the parabolic subgroup of $\mathcal{C}_g$ generated by $m_u$ and $d_\phi$ ($u\in\GL_g(\F)$, $\phi\in\Phi^{(g)}$). With this data, for a type different from $q^E$, we define an Eisenstein series of genus $g$ and variable $\underline{x}=(x_v)_{v\in\F^g}$ as
$$E_g(\underline{x})=\frac{1}{|\mathcal{P}_g\backslash\mathcal{C}_g|}\sum_{\sigma\in \mathcal{P}_g\backslash \mathcal{C}_g} \left(\sum_{v\in\F^g} x_v^N\right)^\sigma .$$
Note that with the right action of the elements $\sigma$ and the assumption on the type and $N$ this is well-defined. When the type is $q^E$, we set
$$E_g(\underline{x})=\frac{1}{|\mathcal{P}_g\backslash\mathcal{C}_g|}\sum_{\sigma\in \mathcal{P}_g\backslash \mathcal{C}_g} \left( x_{(0,\ldots ,0)}^N\right)^\sigma .$$
It is clear from the definition of $m_u$, $d_{\phi}$ and the fact that all second degree characters are pointed ($\phi(0)=0$), that the group $\mathcal{P}_g$ acts trivially on $x_{(0,\ldots ,0)}$. The presence or lack of the sum $\sum_{v\in\F^g}$ in the above formulas is related to the kernel of $\lambda^{(g)}$ (cf. section \ref{sec:Clifford}) and the fact that the elements of the parabolic subgroup of  $\mathfrak{U}(M_g(\F),\Phi^{(g)})$, and thus also of $\mathcal{C}_g$, are defined up to $\ker (\lambda^{(g)} )$ (cf. section \ref{sec:double-coset-decomp}).

\subsection{Connection with Weil's work} 
Even though the main framework of this article is written in the terminology used by Nebe, Rains and Sloane, most notably in the book \cite{NRS06}, we have benefited substantially also from the fundamental works of Weil \cite{Weil1,Weil2}. While it is not necessary to be familiar with Weil's work to understand the proofs of our results, it grants a better understanding of the underlying structure. For this reason 
we decided to compare the above set-up with Weil's. Occasionally, in further parts of the article, we will use Weil's results and notation to clarify some aspects of our findings. The reader who is not familiar with Weil's work is suggested to skip this part for the time being.

The starting point of \cite{Weil1} is a locally compact abelian group $G$. Such a general approach is kept in Chapter I, and only after Chapter II the group $G$ is taken to be a finite dimensional vector space $X$ and the notation becomes additive instead of multiplicative. For us $G=\mathbb{F}_q^g$ is an additive group and $X = \mathbb{F}^g_q$ may be seen as a vector space over $\mathbb{F}_q$. We discuss separately the cases where $q$ is a prime or not. 

We first need to clarify a discrepancy in the terminology used in Weil's work and in \cite{NRS06}. Given a vector space $X$ as above, Weil's definition of a quadratic map on $X$ (see page 172) does not coincide with the notion of a quadratic map in \cite{NRS06} (see section 1.1 in the book), but rather with what in the book is called a homogeneous pointed quadratic map. In the case of even characteristic this means that Weil's notion of a quadratic map agrees with the notion of a pointed quadratic map of \cite{NRS06}. In the case of odd characteristic, however, Weil's quadratic map is never additive, whereas in \cite{NRS06} additivity is allowed. Actually, it is shown there that every pointed quadratic map decomposes as a sum of a homogeneous quadratic map and an additive map. With this in mind, we see that the group $Ps(X)$ in Weil (defined in page 181) coincides with the hyperbolic co-unitary group $\mathfrak{U}(M_g(\F),\Phi^{(g)})$ when $X$ is a finite vector space over $\mathbb{F}^g_2$. When $X$ is over $\mathbb{F}_p$ with $p$ odd, then one needs to augment Weil's definition to allow for additive quadratic maps; then $Ps(X)$ also coincides with the hyperbolic co-unitary group. 

To include in our discussion also the case where $X$ is over $\mathbb{F}_q$ for $q$ not a prime, we should consider the definition of the group $Ps(X/\mathcal{A})$ which Weil gives in page 208. Indeed, if we take (in Weil's notation) $\mathcal{A}=\mathbb{F}_q$ and $k=\mathbb{F}_p$ where $q=p^e$, then we see that the group $Ps(X/\mathbb{F}_q)$ is a subgroup of $Ps(X/\mathbb{F}_p)$ (here $X$ is seen as a vector space over $\mathbb{F}_p$) consisting of elements of the form $(\sigma,f)$ where $\sigma \in \Sp(X/\mathbb{F}_q)$ and $f$ is an $\mathbb{F}_p$-valued quadratic form of a particular kind. Indeed, as Weil explains in pages 207-208, we require that $f$ is of the form $f(x) = \tau (F(x))$ where $\tau : \mathbb{F}_q \rightarrow \mathbb{F}_p$ is the trace map and $F$ is an $\mathbb{F}_q$-valued quadratic map. Looking at the definition of a hyperbolic co-unitary group in the case of $\mathbb{F}_q$, we see that it coincides with Weil's $Ps(X/\mathbb{F}_q)$ if one also allows the map $F$ to be an $\mathbb{F}_q$-quadratic form in the sense of \cite{NRS06}, that is, including also additive forms.  

We now also briefly explain the relation of the hyperbolic co-unitary group with the group $B_0(G)$ in Weil's work. The group $B_0(G)$ consists of elements of the form $(\sigma, \psi)$ with $\psi \in X_2(G)$ (a second order character on $G$) and $\sigma \in \Sp(G)$. By the above it is enough to relate $Ps(X)$ to $B_0(G)$ when $X$ and $G$ are both finite vector spaces over $\mathbb{F}_p$. We simply write $X$ instead of $G$. 
Weil describes a map $\mu : Ps(X) \rightarrow B_0(X)$, which depends on the choice of a fixed character $\chi : \mathbb{F}_p \rightarrow \mathbb{C}^{\times}$. In this paper this character is taken as $\chi(m) = e^{2 \pi i \frac{m}{p}}$ for $ m \in \mathbb{F}_p$. In our case (finite vector spaces over $\mathbb{F}_p$) it is not hard to see that the composition $\chi \circ f$ with a pointed quadratic map $f$ (in the general sense of \cite{NRS06}) gives a bijection between $\mathbb{F}_p$-valued quadratic maps and second order characters. In this way, we see that the map $\mu$ is an isomorphism in our case, and hence we may identify the groups $Ps(X)$ and $B_0(X)$ with the hyperbolic co-unitary group $\mathfrak{U}(M_g(\F),\Phi^{(g)})$. 

We further note that in \cite{NRS06}, the authors used a fixed non-degenerate symmetric form to identify $X$ and its dual. Hence, given the difference in the way the groups act (from the left or from the right), if we write $x =(x_1,x_2,\ldots,x_g, x_{g+1},\ldots, x_{2g}) \in X\times X$, then the space $X^*$ in Weil corresponds to the space spanned by $x_1,\ldots, x_g$ in the notation of \cite{NRS06}. 

The group $B_0(X)$ acts on $L^2(X)$ and in this way Weil obtains a projective representation of the group $B_0(X)$ as unitary operators. This is the group $\mathbf{B}_0(X)$ in Weil and corresponds to the Clifford-Weil group $\mathcal{C}_g$ in our paper. Since $X$ is finite, the vector space $L^2(X)$ can be identified with $V = \{ f: \F^g \rightarrow \C \}$, a vector space of dimension $|\F|^g$, a basis of which is given by the indicator functions (in the above notation) $x_v : \F^g \rightarrow \C$, $x_v(w) = 1$ if $v=w$ and zero otherwise. In particular, via this correspondence, the element $m_u$ above corresponds to the element $\mathbf{d}(u)$ in Weil, and the element $d_{\phi}$ to the element $\mathbf{t}(\phi)$ in Weil. The partial Fourier transforms are also discussed in Weil (see for example \cite[Proposition 8]{Weil1}) and will be used later in the paper.

\section{Doubling method}\label{sec:doubling-method}
For a fixed genus $g$ we consider the map

\[
D: \C[X_{v}: v\in \F^{2g}] \rightarrow \C[x_{v_1}: v_1\in \F^g] \otimes \C[y_{v_2}: v_2\in \F^g],
\]
induced by
\[
D(X_{(v_1,v_2)}) = x_{v_1}y_{v_2}.
\]
This map was considered by Runge in \cite[page 184]{Runge96} and should be seen as the analogue of the doubling embedding in the theory of Siegel modular forms (of course here we give it on functions rather than the space itself). In particular it has the property that for a code $C$ of some length $N$ and belonging to one of the types discussed above, we have
\begin{equation}\label{eq:restriction}
D(\cwe_{2g}(C))(\underline{X}) = \cwe_g(C,\underline{x})\, \cwe_g(C,\underline{y}).
\end{equation}

As we've seen above, the space $\C[X_{v}: v\in \F^{2g}]$ admits an action of the Clifford-Weil group $\mathcal{C}_{2g}$, and similarly there is an action of $\mathcal{C}_g$ on the two copies of $\C[x_{v}: v\in \F^g]$. We now note that there is a homomorphism (not injective) 
\[
\Delta : \mathcal{C}_g \times \mathcal{C}_g  \rightarrow \mathcal{C}_{2g}
\]
such that for $\gamma_1,\gamma_2 \in \mathcal{C}_g$ and $p\in \C[X_{v}: v\in \F^{2g}]$ we have 
\[
(p(\underline{X}))^{\Delta(\gamma_1,\gamma_2)} = D(p(\underline{X}))^{\gamma_1 \times \gamma_2},
\]
where the right hand side is the action of $\mathcal{C}_g \times \mathcal{C}_g$ on $$\C[x_{v_1}: v_1\in \F^g] \otimes \C[y_{v_2}: v_2\in \F^g] .$$ 
This is explained in \cite[page 170, section 22]{Weil1}. Indeed, using the link between our notation and Weil's as explained above, we may take $X = X_1 \times X_2$ where $X= \F^{2g}$ and $X_1=X_2= \F^g$. Then we have a diagonal embedding $Ps(X_1) \times Ps(X_2) \hookrightarrow Ps(X)$ which is compatible with the action of $Ps(X)$ on $L^2(X) = L^2(X_1) \otimes L^2(X_2)$. Here we understand $L^2(X)$ as a tensor product which is compatible with the map of the variables $X_{(v_1,v_2)} = x_{v_1}\otimes y_{v_2}$, $X_{(v_1,v_2)} \in L^2(X)$, $x_{v_1} \in L^2(X_1)$, $y_{v_2} \in L^2(X_2)$ where we have used the interpretation of the variables as functions (as explained above). As it is explained in the last paragraph of \cite[section 22]{Weil1}, we then obtain a map homomorphism $\mathbf{B}_0(X_1) \times \mathbf{B}_{0}(X_2) \rightarrow \mathbf{B}_0(X)$ (this corresponds to $\Delta : \mathcal{C}_g \times \mathcal{C}_g \rightarrow \mathcal{C}_{2g}$ in our notation), which is compatible with the embedding $Ps(X_1) \times Ps(X_2) \hookrightarrow Ps(X)$. 

In the subsections below we will prove the following.
\begin{bthm}
Let $f\in \C [y_v: v\in\F_q^g]_N$ be a cusp form of genus $g$ associated with type $2^E_{II}$, $q_1^E$ or $q^E$. Then
\[
\left( D(E_{2g})(\underline{x},\underline{y}),f(\underline{y})\right)_g = cN! \bar{f}(\underline{x}),
\]
where the constant $c$ depends on the type but is independent of $f$:
\[
c=\begin{cases}
2^{g^2+2g-Ng/2}\prod_{i=1}^g\frac{2^i-1}{2^{g+i}+1} , & \mbox{for type } 2^E_{II}\\
q^{g^2+2g-Ng/2}\prod_{i=1}^g\frac{q^i-1}{q^{g+i}+1} , & \mbox{for type } q^E_{1}\\
q^{g^2-Ng/2}\prod_{i=1}^g\frac{q^i-1}{q^{g+i}+1} , & \mbox{for type } q^E
\end{cases}
\]
and $\bar{f}(\underline{x})$ denotes the polynomial obtained from $f(\underline{x})$ by applying complex conjugation to the coefficients. 
\end{bthm}

\noindent In subsection \ref{sec:computation} we give a computational evidence towards existence of such constant for the type $2^E_I$.

\subsection{Double coset decomposition}\label{sec:double-coset-decomp}
Using the notation above, our first aim is to find a ``nice'' decomposition of the double coset $\mathcal{P}_{2g} \setminus \mathcal{C}_{2g} / \Delta(\mathcal{C}_g \times \mathcal{C}_g)$. For this we will first find a similar decomposition on the level of hyperbolic co-unitary groups. That is, we will find a decomposition of the double cosets
\[
\mathfrak{P}(M_{2g}(\F), \Phi^{(2g)})\setminus \mathfrak{U}(M_{2g}(\F),\Phi^{(2g)})/ \mathfrak{U}(M_g(\F),\Phi^{(g)}) \times \mathfrak{U}(M_g(\F),\Phi^{(g)})
\]
where $\mathfrak{P}(M_{2g}(\F), \Phi^{(2g)})$ is the hyperbolic unitary Siegel parabolic defined in \cite[pages 130, 136]{NRS06} and is generated by:
$$d(u,\phi) =\left(\left(\begin{matrix} \T{u}^{-1} & \T{u}^{-1}\lambda(\phi)\\ 0 & u\end{matrix}\right) ,\left(\begin{matrix} 0 & \\ & \phi \end{matrix}\right) \right) ,\qquad u\in\GL_{2g}(\F) ,\, \phi\in \Phi^{(2g)}.$$
The element $d(u,\phi)$ corresponds to $m_u d_\phi$ via the representation $\pi$ from section \ref{sec:Clifford}. The other generators of $\mathfrak{U}(M_{2g}(\F),\Phi^{(2g)})$ are of the form 
$$H_{\iota,u_\iota,v_\iota}=\left(\left(\begin{matrix} 1-\T{\iota} & v_\iota \\ -\T{u_\iota} & 1-\iota\end{matrix}\right) ,\left(\begin{matrix}  0 & -\iota\\ & 0\end{matrix}\right) \right) ,$$
where $\iota=u_\iota v_\iota$ varies over symmetric idempotents in $M_{2g}(\F)$, they correspond to the elements $h_{\iota,u_\iota,v_\iota}$ (see \cite[Theorem 5.3.2]{NRS06}). We note here also that the elements 
\[
\left(\left( \begin{matrix} a & b \\ c & d \end{matrix}\right) , \left(\begin{matrix} \phi_1 & m \\  & \phi_2 \end{matrix}\right)\right) \in \mathfrak{U}(M_{2g}(\F),\Phi^{(2g)})
\]
satisfy the identity (see \cite[Definition 5.2.4]{NRS06})
\begin{equation}\label{eq:U-identity}
\left( \begin{matrix} \T{c} a & \T{c} b \\ \T{d}a - 1 & \T{d}b \end{matrix}\right) = \left(\begin{matrix} \lambda(\phi_1) & m \\ \T{m} & \lambda(\phi_2) \end{matrix}\right) ;
\end{equation}
and the group law for $(\sigma ,f), (\sigma' ,f')\in \mathfrak{U}(M_{2g}(\F),\Phi^{(2g)})$ is given by (see \cite[Definition 5.1.1 and equation (5.2.6)]{NRS06})
\begin{equation}\label{eq:operation in U}
(\sigma ,f)(\sigma' ,f')=(\sigma\sigma' ,f[\sigma']+f') ,
\end{equation}
where for $f=\begin{pmatrix} \phi_1 & m\\ & \phi_2\end{pmatrix}\in \Phi^{(2g)}$ we have
\begin{equation}\label{eq:bracket}
\begin{pmatrix} \phi_1 & m\\ & \phi_2\end{pmatrix}\left[ \begin{pmatrix} a & b\\ c & d \end{pmatrix}\right] = \begin{pmatrix} \phi_1' & m'\\ & \phi_2'\end{pmatrix} ,
\end{equation}
and
\begin{eqnarray*}
\phi_1' & = &\phi_1[a]+\phi_2[c] +\{\!\{ \T{a}mc\}\!\} ,\\
m' & = & \T{a}\lambda(\phi_1)b +\T{a}md +\T{c}\lambda(\phi_2)d +\T{c}\T{m}b ,\\
\phi_2' & = &\phi_1[b]+\phi_2[d] +\{\!\{ \T{b}md\}\!\} .
\end{eqnarray*}
For a precise definition of the above symbols we refer the reader to \cite{NRS06}, especially chapter 1. Here we only note that $\lambda(\phi_i)=\lambda^{(g)}(\phi_i)\in  M_{g}(\F)$ is a bilinear form associated to $\phi_i\in \Phi^{(g)}$, $\{\!\{ A\}\!\}$ is an element in $\Phi^{(2g)}$ associated to $A\in M_{2g}(\F)$, the definition of $\phi_i[\; ]$ is given by the formula \eqref{eq:bracket} and $\phi_i[1]=\phi_i$. In particular, $(\sigma ,f)(1 ,f')=(\sigma,f+f')$.

In order to find the aforementioned decomposition, we relate the group $\mathfrak{U}(M_{2g}(\F),\Phi^{(2g)})$ with $\mathcal{G}_{2g}$ (recall Table \ref{table:parameters}). Namely, if we denote by $P_{2g}(\F)\subset \Sp_{2g}(\F)$ the classical Siegel-parabolic (lower left block entry is zero) of size $4g$, then the projection from $\mathfrak{P}(M_{2g}(\F), \Phi^{(2g)})$ to the first coordinate induces an isomorphism
$$\mathfrak{P}(M_{2g}(\F), \Phi^{(2g)}) \cong \ker(\lambda^{(2g)}). (P_{2g}(\F) \cap\mathcal{G}_{2g}).$$
In particular, for type $q^E$ we get $\mathfrak{P}(M_{2g}(\F), \Phi^{(2g)}) \cong P_{2g}(\F)$.

\begin{lem}\label{lem:double-coset-decomp}
Let $\mathfrak{U}(M_{2g}(\F),\Phi^{(2g)})$ be a hyperbolic co-unitary group associated with a code of type $2^E_{II}$, $q^E$ or $q^E_1$, $g\geq 1$. Then
\[
\mathfrak{U}(M_{2g}(\F),\Phi^{(2g)}) = \bigsqcup_{0 \leq r \leq g} \mathfrak{P}(M_{2g}(\F), \Phi^{(2g)})\,\, \boldsymbol{\tau}_r \,\,\left(\mathfrak{U}(M_g(\F),\Phi^{(g)}) \times \mathfrak{U}(M_g(\F),\Phi^{(g)})\right) ,
\]
where 
$$\boldsymbol{\tau}_r = \left( \begin{pmatrix} 1_g & & & \\ & 1_g & & \\ & e_r & 1_g & \\ e_r & & & 1_g\end{pmatrix} ,
 \begin{pmatrix} & e_r &  & \\ 0_g & & & \\ & & 0_g & \\ &  & & 0_g
 \end{pmatrix}\right),\qquad e_r=\left(\begin{smallmatrix} 1_r & \\ & 0_{g-r}\end{smallmatrix}\right) \in M_{g}(\F),$$
$r\in\{ 0,\ldots ,g\}$.
\end{lem}
\begin{proof}
The above discussion implies that the first two rows below are exact.
\[\footnotesize{
\xymatrix{
  0 \ar[r] & \ker\lambda^{(g)} \ar@{^{(}->}[d] \ar[r] & \mathfrak{P}(M_{g}(\F), \Phi^{(g)}) \ar@{^{(}->}[d] \ar[r] & P_{g}(\F) \ar@{^{(}->}[d] \ar[r] & 1 \\
  0 \ar[r] & \ker\lambda^{(g)}\oplus\ker\lambda^{(g)} \ar[r] \ar[d] & \mathfrak{U}(M_{g}(\F),\Phi^{(g)}) \ar[r]^-\phi \ar[d] & \Sp_{2g}(\F) \ar[r] \ar[d] & 1 \\
  & \ker\lambda^{(g)} \setminus (\ker\lambda^{(g)}\oplus\ker\lambda^{(g)}) & \mathfrak{P}(M_{g}(\F), \Phi^{(g)})\backslash \mathfrak{U}(M_{g}(\F),\Phi^{(g)})  &  P_{g}(\F)\backslash\Sp_{2g}(\F) &
}}
\]
By diagram chasing we may write the sequence
$$0\to \ker\lambda^{(g)} \setminus (\ker\lambda^{(g)}\oplus\ker\lambda^{(g)}) \to \mathfrak{P}(M_{g}(\F), \Phi^{(g)})\backslash \mathfrak{U}(M_{g}(\F),\Phi^{(g)}) \to P_{g}(\F)\backslash\Sp_{2g}(\F) \to 1.$$
This is a short exact sequence of pointed sets whose distinguished points are the cosets which include the identity of the corresponding group. That is, if we write $\pi$ for the third arrow, then the preimage $\pi^{-1}(\mathbf{1})$ is in bijection with the image of  $\ker\lambda^{(g)} \cong \ker\lambda^{(g)} \setminus (\ker\lambda^{(g)}\oplus\ker\lambda^{(g)})$ in $\mathfrak{P}(M_{g}(\F), \Phi^{(g)})\backslash \mathfrak{U}(M_{g}(\F),\Phi^{(g)})$. 
Moreover, consider $\mathfrak{P}(M_{g}(\F),\Phi^{(g)})y_1, \mathfrak{P}(M_{g}(\F),\Phi^{(g)}) y_2 \in \mathfrak{P}(M_{g}(\F), \Phi^{(g)})\backslash \mathfrak{U}(M_{g}(\F),\Phi^{(g)})$ with the same image under $\pi$, say $ P_{g}(\F) x \in P_{g}(\F)\backslash\Sp_{2g}(\F)$. Then $\phi(y_1) = p_1 x$ for some $p_1 \in P_{g}(\F)$ and $\phi(y_2) = p_2 x$ for some $p_2 \in P_{g}(\F)$. Since $\phi$ is a group homomorphism, 
$\phi(y_1y_2^{-1}) = \phi(y_1) \phi(y_2)^{-1}= p_1 p_2^{-1} \in P_{g}(\F)$. Hence $\mathfrak{P}(M_{g}(\F),\Phi^{(g)}) y_1 y_2^{-1} \in \ker \pi$ and so it belongs to the image of $\ker\lambda^{(g)} \setminus (\ker\lambda^{(g)}\oplus\ker\lambda^{(g)})$.  

We now consider the right group action of $\Sp_{2g}(\F)\times\Sp_{2g}(\F)$ on the homogeneous space 
$P_{2g}(\F)\backslash\Sp_{4g}(\F)$
induced by the embedding
\[
\Sp_{2g}(\F)\times\Sp_{2g}(\F)\hookrightarrow\Sp_{4g}(\F)
\]
given by $$\left(\left(\begin{smallmatrix} a & b \\ c & d \end{smallmatrix}\right) ,
\left(\begin{smallmatrix} a' & b' \\ c' & d' \end{smallmatrix}\right) \right) \mapsto
\left(\begin{smallmatrix} a & & b & \\ & a' & & b'\\ c & & d &\\ & c' & & d'\end{smallmatrix}\right) ,$$
exactly as in \cite{Shimura}. Similarly there is a map
\[
\mathfrak{U}(M_{g}(\F),\Phi^{(g)})\times \mathfrak{U}(M_{g}(\F),\Phi^{(g)}) \rightarrow \mathfrak{U}(M_{2g}(\F),\Phi^{(2g)}) 
\]
as for example it is explained in \cite[page 171]{Weil1}, where $B_0(G)$ is the notation there for the co-unitary group; one can also derive it from the above embedding of $\Sp_{2g}(\F)\times\Sp_{2g}(\F)$ and the property \eqref{eq:U-identity}. In particular, we may lift the previous action to an action of $\mathfrak{U}(M_{g}(\F),\Phi^{(g)})\times \mathfrak{U}(M_{g}(\F),\Phi^{(g)})$ on
$\mathfrak{P}(M_{2g}, \Phi^{(2g)})\backslash \mathfrak{U}(M_{2g}(\F),\Phi^{(2g)})$ such that
\[\footnotesize{
\xymatrix@R=0.1cm{
  0 \ar[r] & (\ker\lambda^{(g)})^2 \times (\ker\lambda^{(g)})^2 \ar[r]^-i & \mathfrak{U}(M_{g}(\F),\Phi^{(g)})\times \mathfrak{U}(M_{g}(\F),\Phi^{(g)}) \ar[r]^-p & \Sp_{2g}(\F)\times\Sp_{2g}(\F) \ar[r] & 1 \\
  & \rotatebox[origin=c]{90}{\scalebox{1.4}{\reflectbox{\mbox{$\acts$}}}} & \rotatebox[origin=c]{90}{\scalebox{1.4}{\reflectbox{\mbox{$\acts$}}}} & \rotatebox[origin=c]{90}{\scalebox{1.4}{\reflectbox{\mbox{$\acts$}}}} \\
  0 \ar[r] & \ker\lambda^{(2g)} \setminus (\ker\lambda^{(2g)}\times\ker\lambda^{(2g)}) \ar[r]^-j & \mathfrak{P}(M_{2g}(\F), \Phi^{(2g)})\backslash\mathfrak{U}(M_{2g}(\F),\Phi^{(2g)}) \ar[r]^-\pi  & P_{2g}(\F)\backslash\Sp_{4g}(\F) \ar[r] & 1 
}}
\]
Here the compatibility of the group actions is to be understood in the following way: if $g \in \mathfrak{U}(M_{g}(\F),\Phi^{(g)})\times \mathfrak{U}(M_{g}(\F),\Phi^{(g)})$ and $ y \in \mathfrak{P}(M_{2g}(\F), \Phi^{(2g)})\backslash\mathfrak{U}(M_{2g}(\F),\Phi^{(2g)}) $, then 
\begin{equation}\label{eq:pi-compatibility}
\pi (y \cdot g) =  \pi(y) \cdot p(g);
\end{equation}
similarly, for $z\in (\ker\lambda^{(g)})^2\backslash (\ker\lambda^{(2g)}\times \ker\lambda^{(2g)})$ and $w\in  (\ker\lambda^{(g)})^4$ we have
\begin{equation}\label{eq:j-compatibility}
j(z + w) = j(z) \cdot i(w).
\end{equation}

Indeed, to see why we have the compatibility \eqref{eq:pi-compatibility}, write $y=\mathfrak{P}(M_{2g}(\F), \Phi^{(2g)})(y_1,y_2)$ and $g=((g_1,g_2),(h_1,h_2))$, where $y_1\in \Sp_{4g}(\F)$, $g_1, h_1\in\Sp_{2g}(\F)$. Then $p(g)=(g_1,h_1)$ and $\pi(y)=P_{2g} y_1$, so using the operation \eqref{eq:operation in U} in $\mathfrak{U}(M_{2g}(\F),\Phi^{(2g)})$, we obtain
\[
\pi (y \cdot g) = \pi (P_{2g} y_1\cdot (g_1,h_1),*) = P_{2g} y_1\cdot (g_1,h_1) = \pi(y) \cdot p(g).
\]

To prove \eqref{eq:j-compatibility}, recall that $\ker\lambda^{(g)}$ is either trivial or isomorphic to $\F^g$. In the first case there's nothing to show, so assume the latter. Let $z=(\ker\lambda^{(g)})^2+z'\in (\ker\lambda^{(g)})^2\backslash(\ker\lambda^{(g)})^4$ with $z'=(z_{1},z_{2})\in (\ker\lambda^{(g)})^2\times(\ker\lambda^{(g)})^2$ and $w=(w_{1},w_{2})\in (\ker\lambda^{(g)})^2\times(\ker\lambda^{(g)})^2$. (Of course, the sum $(\ker\lambda^{(g)})^2+z'$ needs to be understood in a suitable way, depending on the choice of the quotient $(\ker\lambda^{(g)})^2\backslash(\ker\lambda^{(g)})^4$.) Note that
\[
i(w)=(1_{4g},\(\begin{smallmatrix} w_{1} & \\ & w_{2}\end{smallmatrix}\))\quad\mbox{and}\quad
j(z) = (1_{4g},\(\begin{smallmatrix} 0 & \\ & (\ker\lambda^{(g)})^2\end{smallmatrix}\) +\(\begin{smallmatrix} z_{1} & \\ & z_{2}\end{smallmatrix}\)).
\]
Hence
\[
j(z+w)=j((\ker\lambda^{(g)})^2+z'+w) = (1_{4g},\(\begin{smallmatrix} 0 & \\ & (\ker\lambda^{(g)})^2\end{smallmatrix}\) +\(\begin{smallmatrix} z_{1} + w_{1} & \\ & z_{2}+w_{2}\end{smallmatrix}\))
\]
and, by \eqref{eq:operation in U} and \eqref{eq:bracket},
\begin{align*}
j(z) \cdot i(w) &= (1_{4g},\(\begin{smallmatrix} 0 & \\ & (\ker\lambda^{(g)})^2\end{smallmatrix}\) +\(\begin{smallmatrix} z_{1} & \\ & z_{2}\end{smallmatrix}\))\cdot (1_{4g},\(\begin{smallmatrix} w_{1} & \\ & w_{2}\end{smallmatrix}\))\\
&=(1_{4g},\(\begin{smallmatrix} 0 & \\ & (\ker\lambda^{(g)})^2\end{smallmatrix}\) +\(\begin{smallmatrix} z_{1} & \\ & z_{2}\end{smallmatrix}\) +\(\begin{smallmatrix} w_{1} & \\ & w_{2}\end{smallmatrix}\)) =j(z+w).
\end{align*}

Let us now write $\{x_1,\ldots,x_n\}$ for a collection of representatives of the action of $ \Sp_{2g}(\F)\times\Sp_{2g}(\F) $ on $P_{2g}(\F)\backslash\Sp_{4g}(\F)$, and let us further select any lifts $\{y_1,\ldots,y_n\}$ of these $x_i$'s in $\mathfrak{P}(M_{2g}(\F), \Phi^{(2g)})\backslash\mathfrak{U}(M_{2g}(\F),\Phi^{(2g)})$, i.e. $\pi(y_i) = x_i$ for $i\in\{ 1,\ldots ,n\}$. We claim that the $y_i$'s form a set of representatives for the orbits of $\mathfrak{P}(M_{2g}(\F), \Phi^{(2g)})\backslash\mathfrak{U}(M_{2g}(\F),\Phi^{(2g)})$ with respect to the action of $G:=\mathfrak{U}(M_{g}(\F),\Phi^{(g)})\times \mathfrak{U}(M_{g}(\F),\Phi^{(g)})$.

First, we show that the orbits represented by $y_i$'s are disjoint. Suppose there exist $g_1,g_2 \in G$ such that $y_i g_1 = y_jg_2$. Then after taking the map $\pi$ and using the compatibility \eqref{eq:pi-compatibility} we obtain $x_i p(g_1) = x_j p(g_2)$. This means that $x_i$ and $x_j$ are in the same orbit, and therefore $i=j$. Secondly, we show that the orbits represented by $y_i$'s cover the entire space. Fix an element $y \in \mathfrak{P}(M_{2g}(\F), \Phi^{(2g)})\backslash\mathfrak{U}(M_{2g}(\F),\Phi^{(2g)})$ and assume that $\pi(y)$ belongs to the orbit represented by $x_i$. We claim that $y \in y_i G$.

By the definition of $\phi$,   
\[
\phi(\(\begin{matrix} a & b \\ c & d\end{matrix}\) ,\(\begin{matrix} \phi_1 & m \\ & \phi_2\end{matrix}\)) = \(\begin{matrix} a & b \\ c & d\end{matrix}\)
\]
Note here that by \eqref{eq:U-identity} the matrix $\(\begin{smallmatrix} \phi_1 & m \\ & \phi_2\end{smallmatrix}\)$ is defined uniquely by $a, b, c, d$, up to $\(\begin{smallmatrix} \ell_1 & \\ & \ell_2\end{smallmatrix}\)$ with $\ell_1,\ell_2\in\ker\lambda$.
Observe also that
\[
(\(\begin{matrix} p_a & p_b \\ & p_d\end{matrix}\) ,\(\begin{matrix} 0 & \\ & f\end{matrix}\) ) (\(\begin{matrix} a & b \\ c & d\end{matrix}\) ,\(\begin{matrix} \phi_1 & m \\ & \phi_2\end{matrix}\) )
=(\(\begin{matrix} p_a & p_b \\ & p_d\end{matrix}\)\(\begin{matrix} a & b \\ c & d\end{matrix}\) ,\(\begin{matrix} f[c] & \\ & f[d] \end{matrix}\) +\(\begin{matrix} \phi_1 & m \\ & \phi_2\end{matrix}\) ).
\]
Hence $\pi^{-1}(P_{2g} \(\begin{smallmatrix} a & b \\ c & d\end{smallmatrix}\))$ equals
\[
 \{(p 
\(\begin{matrix} a & b \\ c & d\end{matrix}\) , \(\begin{matrix} (f+\ell )[c] & \\ & (f+\ell )[d] \end{matrix}\) +\(\begin{matrix} \phi_1 +\ell_1 & m \\ & \phi_2+\ell_2\end{matrix}\) ) : p\in P_{2g}, \ell,\ell_1,\ell_2\in\ker\lambda^{(2g)}\} ,
\]
where $\phi_1, m, \phi_2$ satisfy \eqref{eq:U-identity} and, similarly, $f$ is determined by $p$ up to $\ell\in\ker\lambda^{(2g)}$. This implies that from the equality $\pi(y)=\pi(y_i)$ for $y=\mathfrak{P}(M_{2g}(\F), \Phi^{(2g)})(y_1,y_2)$ and $y_i=\mathfrak{P}(M_{2g}(\F), \Phi^{(2g)})(y_{i,1},y_{i,2})$ with $y_{i,1}=\(\begin{smallmatrix} a & b \\ c & d\end{smallmatrix}\)$, we may deduce that $y_1=py_{i,1}$ and $y_2=y_{i,2}+\(\begin{smallmatrix} \ell_1+(f+\ell )[c] & \\ & \ell_2+(f+\ell )[d] \end{smallmatrix}\)$ for some $p\in P_{2g}$ and $\ell_1,\ell_2,\ell\in\ker\lambda^{(2g)}$. Hence
\[
y=\mathfrak{P}(M_{2g}(\F), \Phi^{(2g)})(y_1,y_2)=\mathfrak{P}(M_{2g}(\F), \Phi^{(2g)})(p,\(\begin{matrix} 0 & \\ & f+\ell \end{matrix}\))(y_{i,1} , y_{i,2})(1_{4g}, \(\begin{matrix} \ell_1 & \\ & \ell_2 \end{matrix}\)) \in y_i G.
\]

The representatives for the double coset of $\Sp_{4g}(\F)$ are given in \cite[Lemma 4.2]{Shimura}\footnote{The proof is written in a situation when  $\F$ is a number field, but it also holds for finite fields.}, these are 
$$\tau_r =\left(\begin{smallmatrix} 1_g & & & \\ & 1_g & & \\ & e_r & 1_g & \\ e_r & & & 1_g\end{smallmatrix}\right),\qquad\mbox{where}\quad e_r=\left(\begin{smallmatrix} 1_r & \\ & 0_{g-r}\end{smallmatrix}\right) \in M_{g}(\F) ,\quad r\in\{ 0,\ldots ,g\} .$$ 
Lifting these representatives (see \cite[Definition 5.2.4 and section 1.10]{NRS06}), we obtain representatives for the double coset of $\mathfrak{U}(M_{2g}(\F),\Phi^{(2g)})$ of the form
$$( \begin{pmatrix} 1_g & & & \\ & 1_g & & \\ & e_r & 1_g & \\ e_r & & & 1_g\end{pmatrix} ,
 \begin{pmatrix} & e_r &  & \\ 0_g & & & \\ & & 0_g & \\ &  & & 0_g
 \end{pmatrix})$$
where $r\in\{ 0,\ldots ,g\}$. 
\end{proof}

\begin{rem}
The first part of the proof also works for type $2^E_I$. The missing ingredient is the double coset decomposition for the group $O^+_{4g}(\F_2)$.
\end{rem}

\begin{rem}
As we mentioned earlier, when $\F$ is of odd characteristic, the group $A.B$ is the semidirect product $A\rtimes B$. Therefore in this case we could infer the hypothesis of the Lemma from \cite[Lemma 5.1]{BouganisMarzec}. 
\end{rem}

Now, following \cite[sections $46 - 47$]{Weil1}, we can write the representatives $\boldsymbol{\tau}_r$ in terms of generators of the groups $\mathfrak{U}(M_{2j}(\F),\Phi^{(2j)})$ for $j\in\{ r,g-r,g\}$. However, we will still use the group law \eqref{eq:operation in U}. Let
\[
 \begin{pmatrix} a_1 & a_2 & b_1 & b_2 \\
                a_3 & a_4 &  b_3 & b_4 \\
                c_1 & c_2 & d_1 & d_2 \\
                c_3 & c_4 & d_3 & d_4
      \end{pmatrix}\tilde{\otimes}
 \begin{pmatrix} a_1' & a_2' & b_1' & b_2' \\
                a_3' & a_4' &  b_3' & b_4' \\
                c_1' & c_2' & d_1' & d_2' \\
                c_3' & c_4' & d_3' & d_4'
      \end{pmatrix}   =
 \begin{pmatrix}
  a_1 & & a_2 & & b_1 & & b_2 & \\
 & a_1' & & a_2' & & b_1' & & b_2' \\
 a_3 & & a_4 & &  b_3 & & b_4 & \\
              &  a_3' & & a_4' &  & b_3' & & b_4' \\
                 c_1 & & c_2 & & d_1 & & d_2 & \\               
               & c_1' & & c_2' & & d_1' & & d_2' \\
                 c_3 & & c_4 & & d_3 & & d_4 & \\
               & c_3' & & c_4' & & d_3' & & d_4'
      \end{pmatrix}      
\]
where the elements $x\in\{ a_i,b_i,c_i,d_i: i\in\{ 1,2,3,4\}\}$ are square matrices of size $r$, and $x'$ are of size $g-r$.
We consider an embedding 
\[
\mathfrak{U}(M_{2r}(\F),\Phi^{(2r)}) \times \mathfrak{U}(M_{2g-2r}(\F),\Phi^{(2g-2r)}) \hookrightarrow \mathfrak{U}(M_{2g}(\F),\Phi^{(2g)})
\]
\[
((\sigma ,f), (\sigma',f')) \mapsto (\sigma ,f) \otimes (\sigma',f'),\quad\mbox{where}\quad
(\sigma ,f) \otimes (\sigma',f') =(\sigma \,\tilde{\otimes}\,\sigma', f\,\tilde{\otimes}\, f') .
\]
In particular, if $(\sigma',f')=(1_{4g-4r}, 0_{4g-4r})$ is the identity element, we obtain an embedding of $\mathfrak{U}(M_{2r}(\F),\Phi^{(2r)})$ into $\mathfrak{U}(M_{2g}(\F),\Phi^{(2g)})$.

\begin{lem}\label{lem:decomposition}
The element $\boldsymbol{\tau}_r\in \mathfrak{U}(M_{2g}(\F),\Phi^{(2g)})$ is equal to the product
\begin{multline*}
(\begin{pmatrix} 1_g & & & e_r \\ & 1_g & e_r & \\ &  & 1_g & \\ & & & 1_g
\end{pmatrix} , 
\begin{pmatrix} 0_g & & & \\ & 0_g & & \\ &  & & e_r \\ & & 0_g & 
\end{pmatrix} )\\ \left(
( \begin{pmatrix} & & & -1_r \\ &  & -1_r & \\ & 1_r &  & \\ 1_r & & & 
\end{pmatrix} ,
\begin{pmatrix} & & -1_r &  \\ &  &  & -1_r \\ 0_r & & & \\ & 0_r & & 
\end{pmatrix} )
 ( \begin{pmatrix} 1_r & & & 1_r \\ & 1_r & 1_r & \\ &  & 1_r & \\ & & & 1_r
\end{pmatrix} ,
\begin{pmatrix} 0_r & & & \\ & 0_r & & \\ &  & & 1_r \\ & & 0_r & 
\end{pmatrix} )\right.\\
\left. \otimes ( \begin{pmatrix} 1_{g-r} & & & \\ &  1_{g-r} & & \\ & &  1_{g-r} & \\ &  & &  1_{g-r} \end{pmatrix} , 0_{4g-4r} )\right) ,
\end{multline*}
that is,
\begin{equation*}
\boldsymbol{\tau}_r=d(1_{2g},\(\begin{smallmatrix} & e_r \\ 0_g &   
\end{smallmatrix}\))\left( H_{1_{2r},-\nu,\nu}\, d((1_{2r},\(\begin{smallmatrix} & 1_r \\ 0_r &   
\end{smallmatrix}\))\otimes d(1_{2g-2r},0_{2g-2r})\right)
\end{equation*}
where $\nu=\(\begin{smallmatrix} & 1_r \\ 1_r &  \end{smallmatrix}\)$.
\end{lem}

\begin{proof}
This is a direct computation. We decided to include it here to help the reader understand quite general formulas given in \eqref{eq:bracket}. As an example we compute the product of the two elements in the second line. We claim that it is  
\begin{equation}\label{proof2:matrix} 
(\begin{pmatrix} & & & -1_r \\ &  & -1_r & \\ & 1_r & 1_r & \\ 1_r & & & 1_r
\end{pmatrix} ,
\begin{pmatrix} 0_g & & -1_r &  \\ & 0_g &  & -1_r \\  & & 0_r & -1_r \\ &  & & 0_r
\end{pmatrix} ) .
\end{equation}
By the formula \eqref{eq:bracket}, it suffices to verify that 
$$\begin{pmatrix} & & -1_r &  \\ &  &  & -1_r \\ 0_r & & & \\ & 0_r & & 
\end{pmatrix}\left[ \begin{pmatrix} 1_r & & & 1_r \\ & 1_r & 1_r & \\ &  & 1_r & \\ & & & 1_r
\end{pmatrix} \right] = \begin{pmatrix} 0_r & & -1_r &  \\ & 0_r &  & -1_r \\  & & 0_r & -2\, 1_r \\ &  & & 0_r
\end{pmatrix} .$$
We see that $\phi_1'=0_{2r}$ because $\phi_1=\phi_2=c=0_{2r}$, but
$$\phi_2'=0_{2r}+0_{2r}+\{\!\{ \begin{pmatrix}  & 1_r \\  1_r & \end{pmatrix} \begin{pmatrix}  -1_r & \\ & -1_r \end{pmatrix}\begin{pmatrix} 1_r & \\ & 1_r \end{pmatrix}\}\!\} =\{\!\{ \begin{pmatrix}  & -1_r \\  -1_r & \end{pmatrix}\}\!\} =\begin{pmatrix} 0_r & -2\, 1_r \\ &  0_r
\end{pmatrix}$$
(see \cite[Chapter 1.10]{NRS06} and take $\tau$ to be the transposition) and $m'=\begin{pmatrix}  -1_r & \\ & -1_r \end{pmatrix}$ because $\lambda(0_{2r})=0_{2r}$.
Now, taking the embedding, that is, applying $\otimes (1_{4g-4r},0_{4g-4r})$ to the element \eqref{proof2:matrix}, we obtain 
$$( \begin{pmatrix} e_r' & & & -e_r \\ & e_r' & -e_r & \\ & e_r & 1_g & \\ e_r & & & 1_g
\end{pmatrix} ,
\begin{pmatrix}  0_g & & -e_r &  \\ &  0_g &  & -e_r \\ & & 0_g & -e_r \\ &  &  &  0_g \end{pmatrix} ),$$
where $e_r'=\left(\begin{smallmatrix} 0_r & \\ & 1_{g-r}\end{smallmatrix}\right) =1_g -e_r\in M_{g}(\F)$.

We leave it to the reader to verify that when we multiply the above element on the left by the remaining element from the parabolic subgroup, we obtain $\boldsymbol{\tau}_r$.
\end{proof}

Passing to the Clifford-Weil group $\mathcal{C}_{2g}$, we see that $\pi(\boldsymbol{\tau}_r)$ acts on $X=\F^{2g}$ via
\begin{equation}\label{eq:Clifford-decomposition-to-gens}
\pi(\boldsymbol{\tau}_r)=d_\phi\left( h_{1_{2r},-\nu ,\nu} d_{\phi_r} \otimes d_{0_{2g-2r}}\right) ,
\end{equation}
where $\phi=\(\begin{smallmatrix} & e_r \\ 0_g &   
\end{smallmatrix}\)$, $\phi_r=\(\begin{smallmatrix} & 1_r \\ 0_r &   
\end{smallmatrix}\)$. 
More precisely: we write $X=\F^{2g}$ as $X=X_1\oplus X_2$, where $X_1=\F^r\times\{0\}^{g-r}\times \F^r\times\{ 0\}^{g-r}\cong\F^{2r}$ and $X_2$ is its complement in $X$ (such a decomposition is dictated by the image of the lower left corner of $\tau_r$ on $X$, i.e. the image of the map $X\ni x\mapsto x\(\begin{smallmatrix} & e_r \\ e_r &   
\end{smallmatrix}\)$). The operation $\otimes$ is defined so that $h_{1_{2r},-\nu ,\nu} d_{\phi_r}$ acts on $X_1$ and $d_{0_{2g-2r}}=\mathrm{id}_{|_{X_2}}$. More generally, for $x=x_1+x_2\in X_1\oplus X_2$ we define
$$x^{A\otimes B}=x_1^A\oplus x_2^B,$$ 
where the action is given by \eqref{generators}.

\subsection{Decomposition of the Eisenstein series}\label{sec:dec Eis series}
Lemma \ref{lem:decomposition} allows us to decompose the Siegel-type Eisenstein series $E_{2g}$ into summands indexed by $r\in\{ 0,\ldots ,g\}$. More precisely,
\[
E_{2g}(\underline{X}) =\frac{1}{|\mathcal{P}_{2g}\backslash\mathcal{C}_{2g}|}\sum_{\sigma\in \mathcal{P}_{2g}\backslash \mathcal{C}_{2g}} \left(\sum_{v\in\F^{2g}} X_v^N\right)^\sigma =  \frac{1}{|\mathcal{P}_{2g}\backslash\mathcal{C}_{2g}|} \sum_{\sigma\in \mathfrak{P}_{2g}\backslash \mathfrak{U}_{2g}} \left(\sum_{v\in\F^{2g}} X_v^N\right)^{\pi(\sigma)}
\]
where $\mathfrak{P}_{2g} = \mathfrak{P}(M_{2g}(\F),\Phi^{(2g)})$, $\mathfrak{U}_{2g} = \mathfrak{U}(M_{2g}(\F),\Phi^{(2g)})$, may be written as
\[ 
E_{2g}(\underline{X})=\frac{1}{|\mathcal{P}_{2g}\backslash\mathcal{C}_{2g}|} \sum_{r=0}^g E^{(r)}(\underline{X}),
\] 
where
\[ 
E^{(r)}(\underline{X}) = \sum_{\sigma} \left(\sum_{v\in\F^{2g}} X_v^N\right)^{\pi(\sigma)},\qquad\sigma \in \mathfrak{P}_{2g} \setminus \mathfrak{P}_{2g} \boldsymbol{\tau}_r \left(\mathfrak{U}_g \times \mathfrak{U}_g\right) . 
\]
The decomposition of $E_{2g}$ in case of type $q^E$ is analogous.\newline

We now introduce the notion of a Klingen-type parabolic subgroup $\mathfrak{P}^r_{g} \subseteq \mathfrak{U}_g$ for $r=0,\ldots g$, where $\mathfrak{P}^0_{g} = \mathfrak{P}_{g}$ and $\mathfrak{P}^g_{g} = \mathfrak{U}_g$. We first recall the classical maximal $P_g^r$ parabolics in $\Sp_g(\F)$ for $r\in\{ 0,\ldots, g\}$ (see \cite[page 545]{Shimura}). 
For a fixed $r$ we write every element in $\Sp_g(\F)$ as 
\[
 \begin{pmatrix} a_1 & a_2 & b_1 & b_2 \\
                a_3 & a_4 &  b_3 & b_4 \\
                c_1 & c_2 & d_1 & d_2 \\
                c_3 & c_4 & d_3 & d_4
      \end{pmatrix}
\]
where $a_1$ is a square matrix of size $r$ and $a_4$ of size $g-r$, and similarly for the other matrices. Then we define $P_g^r$ to be the subgroup of $\Sp_g$ consisting of elements of the form
\[
 \begin{pmatrix} a_1 & 0 & b_1 & b_2 \\
                a_3 & a_4 &  b_3 & b_4 \\
                c_1 & 0 & d_1 & d_2 \\
                0 & 0 & 0 & d_4
      \end{pmatrix}
\]
In particular, $P_g^g = \Sp_g$ and $P_g^0 = P_g$ in the notation we used above. We now extend this definition to the hyperbolic co-unitary group. Using the fact that 
$$\mathfrak{U}(M_g(\F),\Phi^{(g)})\cong (\ker\lambda^{(g)}\oplus\ker\lambda^{(g)} ).\mathcal{G}_g,$$
we define $\mathfrak{P}_g^r$ as the inverse image of $P_g^r$ in $\mathfrak{U}_g$. Then 
\[
\mathfrak{P}_g^r \cong M.P_g^r
\]
for some subgroup $M$ of $\ker\lambda^{(g)}\oplus\ker\lambda^{(g)}$. Of course, in the case $q^E$ we have $\mathfrak{P}_g^r \cong P_g^r$ and in the case $q^E_{1}$ we have that  $M= \F^r \oplus F^g$ (see \cite{BouganisMarzec}).

The following Lemma can be shown similarly to \cite[Lemma 5.2]{BouganisMarzec}.
\begin{lem}With notation as above we have
\[
\mathfrak{P}_{2g} \boldsymbol{\tau}_r \,\,\left(\mathfrak{U}_g \times \mathfrak{U}_g\right) = \bigsqcup_{\beta,\gamma,\xi} \mathfrak{P}_{2g} \boldsymbol{\tau}_r ((\beta \times (\xi \times 1_{2g-2r}) \gamma)=\bigsqcup_{\beta,\gamma,\xi} \mathfrak{P}_{2g} \boldsymbol{\tau}_r ( (\xi \times 1_{2g-2r}) \beta \times \gamma),
\]
where $\xi$ runs over $\mathfrak{U}_r$, $\beta, \gamma$ over $\mathfrak{P}^r_{g} \setminus \mathfrak{U}_g$. 
\end{lem}

Using this Lemma we can now write 
\[
E^{(r)}(\underline{X}) = \sum_{\beta,\gamma,\xi} \left(\sum_{v\in\F^{2g}} X_v^N\right)^{\pi(\boldsymbol{\tau}_r)\pi(\sigma_{\xi,\beta,\gamma})}
\]
where $\sigma_{\xi,\beta,\gamma} = (\xi \times 1_{2g-2r}) \beta \times \gamma$ and $\pi(\boldsymbol{\tau}_r) =d_\phi\left( h_{1_{2r},-\nu ,\nu} d_{\phi_r} \otimes d_{0_{2g-2r}}\right)$. A similar identity holds in case of type $q^E$.

\subsection{The case $r<g$}\label{sec:r<g}
Recall that the map 
\[
D: \C[X_{v}: v\in \F^{2g}] \rightarrow \C[x_{v_1}: v_1\in \F^g] \times \C[y_{v_2}: v_2\in \F^g],
\]
is induced by $D(X_{v_1v_2}) = x_{v_1}y_{v_2}$. 

\begin{prop}\label{prop:r<g}
Let $f\in\C[y_{v_2}: v_2\in \F^{g}]_N$ be a cusp form of genus $g$. Then taking the inner product in the $\underline{y}$ variable we have that 
\[
\left( D(E^{(r)}), f\right)_g = 0 
\]
for all $r < g$.
\end{prop}
The proof of this proposition will take the rest of this section, but before we can start with the proof we need to establish some facts about the notion of a cusp form and its relation to Siegel's $\Phi$-operator.\newline

We first generalise the notion of Siegel $\Phi$-operator presented in section \ref{sec:cusp forms} after \cite{N09}. For $0 \leq j \leq g$ and $\mathbf{w} \in \F^{j}$ we define 
\[
\Phi_{g,j}^{(\mathbf{w})} : \C[x_v: v\in \F^g] \rightarrow \C[x_v: v\in \F^{g-j}]
\]
by setting
\[
\Phi_{g,j}^{(\mathbf{w})} (x_{(v_1,\ldots,v_g)}) = \begin{cases}
    x_{(v_1,\ldots,v_{g-j})}, & \text{if}\,\,\, (v_{g-j+1}, \ldots,v_g) = \mathbf{w}\\
    0, & \text{otherwise}
\end{cases}
\]
and extending it to a ring homomorphism. Similarly we define the lift operator 
\[
\phi_{g,j}^{(\mathbf{w})} (x_{(v_1,v_2,\ldots,v_{g-j})}) := x_{(v_1,v_2,\ldots,v_{g-j}, \mathbf{w})}.
\]
Then
$\Phi_{g,j}^{(\mathbf{w})} \circ \phi_{g,j}^{(\mathbf{w})}$ is the identity map on the space $\C[x_v: v\in \F^{g-j}]$. The case discussed in \cite{N09} corresponds to $\mathbf{w}$ being the all zero word. For homogeneous polynomials $p \in \C[x_v: v\in \F^{g}]_N$ and $q \in \C[x_v: v\in \F^{g-j}]_N$ of degree $N$ we still have the equality of the inner products:
\[
(\phi_{g,j}^{(\mathbf{w})}(q), p)_g = (q, \Phi_{g,j}^{(\mathbf{w})}(p))_{g-j}.
\]
Indeed, it is clear from the definition of the inner product that an orthogonal basis of the underlying space consists of monomials of degree $N$, and for them the equality holds.

\begin{lem}
If $f\in\C[x_v:v\in\F_q^g]_N$ is a cusp form of genus $g$, then $\Phi_{g,j}^{(\mathbf{w})}(f) = 0$ for all $\mathbf{w} \in \F_q^{j}$ and all $j\in\{ 1,\ldots ,g\}$.
\end{lem}
\begin{proof}
By the assumption, $f$ is invariant under the action of the corresponding Clifford-Weil group and $\Phi_{g,1}(f) = 0$. That is, 
$$f(\underline{x})=\sum_{\substack{k_0,\ldots ,k_{q^g-1}\geq 0\\ k_0+\ldots +k_{q^g-1}=N}} a_{k_0,\ldots ,k_{q^g-1}}\prod_{\alpha\in\F_q^g} x_{\alpha}^{k_{\alpha}}$$
where $\Phi_{g,1}\left( \prod_{\alpha\in\F_q^g} x_{\alpha}^{k_{\alpha}}\right) = 0$ whenever $a_{k_0,\ldots ,k_{q^g-1}}\neq 0$. More precisely, in each term $\prod_{\alpha\in\F_q^g} x_{\alpha}^{k_{\alpha}}$ there is a variable $x_{\alpha}$ where $\alpha \in \F_q^g$ is of the form $(*,*,\ldots,*,a)$ for some $a \neq 0$.

It suffices to prove the hypothesis for $j=1$ because $\Phi^{(\mathbf{w})}_{g,j}=\Phi^{(w_j)}_{g-j+1,1}\circ\ldots\circ\Phi^{(w_2)}_{g-1,1} \circ \Phi^{(w_1)}_{g,1}$ for $\mathbf{w}=(w_j,\ldots ,w_2,w_1)\in \F_q^{j}$, $j\in\{ 1,\ldots ,g\}$. 
Suppose that $\Phi_{g,1}^{(w)}(f)\neq 0$ for some $w\in\F_q$. Then 
$\Phi^{(w)}_{g,1}\left( \prod_{\alpha\in\F_q^g} x_{\alpha}^{k_{\alpha}}\right) \neq 0$ for some term $\prod_{\alpha\in\F_q^g} x_{\alpha}^{k_{\alpha}}$. That is, in this term all the variables with non-trivial powers are of the form $(*,*,\ldots,w)$. But since $f$ is invariant under translations (which belong to the parabolic subgroup of $\mathcal{C}_g$), we may translate $f$ by $(*,*,\ldots,-w)$ to obtain another term of $f$ where all the variables are of the form $(*,*,\ldots,0)$. But then $\Phi_{g,1}(f) \neq 0$, contradiction.
\end{proof}

\begin{proof}[Proof of Proposition \ref{prop:r<g}] We first give the proof excluding the case of $q^E$ and at the end we explain how it needs to be modified to accommodate also this case. \newline

Due to the fact that the map $D$ commutes with the homomorphism $\Delta : \mathcal{C}_g \times \mathcal{C}_g \rightarrow \mathcal{C}_{2g}$, we have
\begin{align*}
D(E^{(r)}(\underline{X})) &= D\left( \sum_{\beta,\gamma,\xi} \left(\sum_{v\in\F^{2g}} X_v^N\right)^{\pi(\boldsymbol{\tau}_r)\pi(\sigma_{\xi,\beta,\gamma})}\right)\\
& = \sum_{\beta,\gamma,\xi} \left( D\left(\left(\sum_{v\in\F^{2g}} X_v^N\right)^{\pi(\boldsymbol{\tau}_r)}\right)\right)^{\pi((\xi \times 1_{2g-2r})\beta) \otimes \pi(\gamma)} 
\end{align*}
where
\[ \pi(\boldsymbol{\tau}_r) =d_\phi\left( h_{1_{2r},-\nu ,\nu} d_{\phi_r} \otimes d_{0_{2g-2r}}\right). \]
Recall that the elements $d_{\phi}$ act trivially on $X_v^N$ and in particular they act trivially on $\sum_{v\in\F^{2g}} X_v^N$. Hence 
\begin{equation} \label{general case}
\left(\sum_{v\in\F^{2g}} X_v^N\right)^{\pi(\boldsymbol{\tau}_r)} = 
\left(\sum_{v\in\F^{2g}} X_v^N\right)^{h_{1_{2r},-\nu ,\nu} d_{\phi_r} \otimes d_{0_{2g-2r}}}
\end{equation}
If we write $v = (v_1,v_1',v_2,v_2') \in \F^{2g}$ where $v_1,v_2 \in \F^{r}$ and $v_1',v_2' \in \F^{g-r}$, then $\left( X_v^N\right)^{h_{1_{2r},-\nu ,\nu} d_{\phi_r} \otimes d_{0_{2g-2r}}}$ is equal to
\[
\left(q^{-r}\sum_{w_1, w_2\in\F^{r} }\exp(2\pi i\beta^{(r)}(w_1,w_2))\exp(2\pi i\beta^{(r)}(w_1,v_2))\exp(2\pi i\beta^{(r)}(w_2,v_1)) X_{(w_1, v_1', w_2,v_2')}\right)^N .
\]
Hence
\begin{align*}
&\mathcal{D}:= D\left(\left(\sum_{v\in\F^{2g}} X_v^N\right)^{\pi(\boldsymbol{\tau}_r)}\right)\\
& = q^{-Nr}\sum_{(v_1,v_1',v_2,v_2') \in \F^{2g}}\left(\sum_{w_1, w_2\in\F^{r} }\exp\left( 2\pi i (\beta^{(r)}( w_1,w_2)+\beta^{(r)}(w_1,v_2)+\beta^{(r)}(w_2,v_1))\right) x_{(w_1, v_1')}y_{(w_2,v_2')}\right)^N\\
& = q^{-Nr}\sum_{(v_1',v_2') \in \F^{2g-2r}}\\
& \sum_{(v_1,v_2) \in \F^{2r}} \left(\sum_{w_1, w_2\in\F^{r} }\exp\left( 2\pi i (\beta^{(r)}( w_1,w_2)+\beta^{(r)}(w_1,v_2)+\beta^{(r)}(w_2,v_1))\right) \phi_{g,g-r}^{(v_1')}(x_{w_1})\phi_{g,g-r}^{(v_2')}(y_{w_2})\right)^N
\end{align*}
In particular, as a function of $\underline{y}$, $\mathcal{D}=\sum_{v_2' \in \F^{g-r}}\phi_{g,g-r}^{(v_2')}(P(\underline{y}))$ for the polynomial $P$ (also depending on the variable $\underline{x}$) as in the formula above. Now, taking the inner product with respect to the ``second variable'' $\underline{y}$, we obtain
\[
\left( D(E^{(r)}), f\right)_g = \left(\sum_{\beta,\gamma,\xi} \mathcal{D}^{\pi((\xi \times 1_{2g-2r})\beta ) \otimes \pi(\gamma)}, f \right)_g
 = \sum_{\beta,\xi}\left( \sum_{\gamma} \left( \mathcal{D}, f^{\pi(\gamma)^*} \right)_g \right)^{\pi((\xi \times 1_{2g-2r})\beta)},
\]
where for a matrix $A$, we denote by $A^*$ the complex conjugate transpose. 
But, as we have explained in section \ref{sec:cusp forms}, 
$\pi(\gamma)^* = \pi(\gamma)^{-1} \in \mathcal{C}_g$. Further, since $f$ is invariant under the action of $\mathcal{C}_g$, 
\begin{align*}
\left( D(E^{(r)}), f\right)_g & = |\mathfrak{P}_g^r\backslash\mathfrak{U}_g| \sum_{\beta,\xi} \left(\sum_{\gamma} \left( \mathcal{D}, f\right)_g\right)^{\pi((\xi \times 1_{2g-2r})\beta)}\\
& = |\mathfrak{P}_g^r\backslash\mathfrak{U}_g|\sum_{\beta,\xi} \left(\sum_{\gamma} \sum_{v_2' \in \F^{g-r}} \left( \phi_{g,g-r}^{(v_2')}(P), f\right)_g\right)^{\pi((\xi \times 1_{2g-2r})\beta)}\\
& = |\mathfrak{P}_g^r\backslash\mathfrak{U}_g|\sum_{\beta,\xi} \left(\sum_{\gamma} \sum_{v_2' \in \F^{g-r}} \left( P, \Phi_{g,g-r}^{(v_2')}(f) \right)_r\right)^{\pi((\xi \times 1_{2g-2r})\beta)}\\
& = 0
\end{align*}
by the assumption that $f$ is a cusp form of genus $g$ and $r<g$.

The proof in the case $q^E$ is essentially the same. The only difference comes from the definition of Eisenstein series. Namely one should replace the sum $\sum_{v \in \F^{2g}} X_v$ with the simple monomial $X_{(0,\ldots ,0)}$ and leave the rest without any change. In particular, the equation $\eqref{general case}$ is still valid because the element $d_{\phi}$ acts trivially on $X_{(0,\ldots ,0)}$, and the vector $v = (v_1,v_1',v_2,v_2') \in \F^{2g}$ just under this equation is simply the zero vector.
\end{proof}

\subsection{The case $r=g$}\label{sec:r=g} In this subsection we establish the following Proposition.
\begin{prop}\label{prop:r=g} Let $f\in\C[y_{v_2}: v_2\in \F_q^{g}]_N$ be a cusp form of genus $g$. If $f$ is associated with type $2^E_{II}$ or $q^E_1$, then
\[
\left( D(E^{(g)}), f\right)_g = q^{2g-Ng/2} N! |\mathfrak{U}_g| \bar{f} ,
\]
where $|\mathfrak{U}_g| = q^{g^2+2g}\prod_{i=1}^{g}(q^{2i}-1)$.
If $f$ is associated with type $q^E$, then
\[
\left( D(E^{(g)}), f\right)_g = q^{-Ng/2} N! |\mathfrak{U}_g| \bar{f} ,
\]
where $|\mathfrak{U}_g| = q^{g^2}\prod_{i=1}^{g}(q^{2i}-1)$.
\end{prop}
\begin{proof}
Assume first that the cusp form and the Eisenstein series are related to type $2^E_{II}$ or to type $q^E_1$. We will make suitable adjustments for the type $q^E$ at the end of the proof.

The first part of the proof of Proposition \ref{prop:r<g} is valid also for $r=g$. In particular,
\[
D(E^{(g)}(\underline{X})) =D\left(\sum_{\xi\in\mathfrak{U}_g} \left(\sum_{v\in\F_q^{2g}} X_v^N\right)^{\pi(\boldsymbol{\tau}_g)\pi(\xi \times 1_{2g})}\right) = \sum_{\xi\in\mathfrak{U}_g}D\left( \left(\sum_{v\in\F_q^{2g}} X_v^N\right)^{\pi(\boldsymbol{\tau}_g)}\right)^{\pi(\xi \times 1_{2g})}
\]
where
\begin{align*}
& D\left(\left(\sum_{v\in\F_q^{2g}} X_v^N \right)^{\pi(\boldsymbol{\tau}_g)}\right)\\
& = q^{-Ng}\sum_{(v_1,v_2) \in \F_q^{2g}}\left(\sum_{w_1, w_2\in\F_q^{g} }\exp\left( 2\pi i (\beta^{(g)}( w_1,w_2)+\beta^{(g)}(w_1,v_2)+\beta^{(g)}(w_2,v_1))\right) x_{w_1}y_{w_2}\right)^N
\end{align*}
For a cusp form $f(\underline{y})$ we will compute
\[
\left( D(E^{(g)}),f\right)_g =\sum_{\xi\in\mathfrak{U}_g} \left( D\left(\left(\sum_{v\in\F_q^{2g}} X_v^N \right)^{\pi(\boldsymbol{\tau}_g)}\right) ,f
\right)_g^{\pi(\xi)} .
\]
First we rearrange the terms:
\begin{multline*}
\left( D\left(\left(\sum_{v\in\F_q^{2g}} X_v^N \right)^{\pi(\boldsymbol{\tau}_g)}\right) ,f
\right)_g\\
= q^{-Ng}\sum_{(v_1,v_2) \in \F_q^{2g}}\left(
\left(\sum_{w_1\in\F_q^{g} } \exp\left( 2\pi i \beta^{(g)}(w_1,v_2)\right) x_{w_1}
\right.\right.\\
\times \left.\left.
\sum_{w_2\in\F_q^{g} }\exp\left( 2\pi i (\beta^{(g)}( w_1,w_2)+\beta^{(g)}(w_2,v_1))\right) y_{w_2}\right)^N ,f\right)_g .
\end{multline*}
Note that for a fixed $v_1$, the function $\chi(w_2)=\exp\left( 2\pi i\beta^{(g)}(w_2,v_1)\right) y_{w_2}$ is an additive character, and thus it is an element of $\ker\lambda^{(g)}\subset\mathcal{C}_g$. Hence the above quantity is equal to
\begin{align*}
& q^{-Ng}\sum_{(v_1,v_2) \in \F_q^{2g}}\left(
\left(\sum_{w_1\in\F_q^{g} } \exp\left( 2\pi i \beta^{(g)}(w_1,v_2)\right) x_{w_1} 
\sum_{w_2\in\F_q^{g} }\exp\left( 2\pi i (\beta^{(g)}( w_1,w_2)\right) y_{w_2}\right)^N ,f^{\chi^*}\right)_g\\
& = q^{-Ng/2}\sum_{(v_1,v_2) \in \F_q^{2g}}\left(
\left(\sum_{w_1\in\F_q^{g} } \exp\left( 2\pi i \beta^{(g)}(w_1,v_2)\right) x_{w_1}
h_{1_{g},1_{g},1_{g}}(y_{w_1})\right)^N,f\right)_g\\
& = q^{-Ng/2}\sum_{(v_1,v_2) \in \F_q^{2g}}\left(
\left(\sum_{w_1\in\F_q^{g} } \exp\left( 2\pi i \beta^{(g)}(w_1,v_2)\right) x_{w_1}
y_{w_1}\right)^N,f\right)_g\\
& = q^{-Ng/2}\sum_{(v_1,v_2) \in \F_q^{2g}}\left(
\left(\sum_{w_1\in\F_q^{g} } x_{w_1}y_{w_1}\right)^N,f^{\chi^*}\right)_g\\
& = q^{2g-Ng/2}\left(\left(\sum_{w\in\F_q^{g} } x_{w}y_{w}\right)^N,f\right)_g
\end{align*}
because $f^{\chi^*}=f$ and the adjoint of the Fourier transform $h_{1_{g},1_{g},1_{g}}\in\mathcal{C}_g$ also belongs to $\mathcal{C}_g$ by which $f$ is invariant. Finally, using multinomial theorem, we obtain further
\[
q^{2g-Ng/2}\sum_{\substack{k_0,\ldots ,k_{q^g-1}\geq 0\\ k_0+\ldots +k_{q^g-1}=N}}\frac{N!}{k_0!\cdots k_{q^g-1}!}\left( \prod_{w\in\F_q^g} x_w^{k_w}y_w^{k_w},f\right)_g = q^{2g-Ng/2} N! \bar{f}(\underline{x}) ,
\]
where we identified the numbers $0, \ldots ,q^g-1$ with $w\in\F_q^g$ via their expansion in base $q$ (extended by a suitable number of zeros to be of length $g$).

In this way,
\[
\left( D(E^{(g)}),f\right)_g =q^{2g-Ng/2}N!\sum_{\xi\in\mathfrak{U}_g} (\bar{f})^{\pi(\xi)} = q^{2g-Ng/2} N!\, |\mathfrak{U}_g| \bar{f} .
\]
We note here that $\bar{f}$ is $\mathcal{C}_g$-invariant. Indeed, since the ring of $\mathcal{C}_g$-invariant polynomials is spanned by $g$-weight enumerators (see \cite[Corollary 5.7.6]{NRS06}) we have $f = \sum_{C} a_C \cwe_g(C)$ for some $a_C \in \mathbb{C}$. But then $\bar{f} = \sum_{C} \bar{a}_C \cwe_g(C)$, since $\cwe_g(C)$ have real coefficients. Hence $\bar{f}$ is also $\mathcal{C}_g$-invariant as linear combination of the weight enumerators. \newline 

The order of the group $\mathfrak{U}_g$ may be computed using the isomorphism \eqref{eq:iso of U}. The order of $\Sp_{2g}(\F)$ as well as the orders of other classical groups may be found for example in \cite{Wilson}.

The proof for the type $q^E$ requires only a small modification. As it was the case in the proof of Proposition \ref{prop:r<g}, it suffices to replace the sum $\sum_{v \in \F^{2g}} X_v$ with the single monomial $X_{(0,\ldots ,0)}$. This simplifies the expression for $D(E^{(g)})$ and gives
\[
\left( D(E^{(g)}),f\right)_g = q^{-Ng/2}\left(\left(\sum_{w\in\F_q^{g} } x_{w}y_{w}\right)^N,f\right)_g .
\]
The rest of the proof remains the same and leads to the hypothesis of the proposition.
\end{proof}

\begin{proof}[Proof of the Theorem]
A discussion at the beginning of section \ref{sec:dec Eis series} together with Propositions \ref{prop:r<g} and \ref{prop:r=g} yield the equality
\[
\left( D(E_{2g}),f\right)_g = \frac{1}{|\mathcal{P}_{2g}\backslash\mathcal{C}_{2g}|} \sum_{r=0}^g \left( E^{(r)},f\right)_g
= \frac{1}{|\mathcal{P}_{2g}\backslash\mathcal{C}_{2g}|} N! \, q^{c_{\mathcal{T}}-Ng/2} |\mathfrak{U}_{g}| \bar{f},
\]
where $|\mathfrak{U}_{g}|=q^{g^2+c_{\mathcal{T}}}\prod_{i=1}^{g}(q^{2i}-1)$, and $q^{c_{\mathcal{T}}}=|\ker\lambda^{(g)}|^2$. 
Because $|P_{2g}(\F_q)|=|\GL_{2g}(\F_q)|\cdot |\{ A\in M_{2g}(\F_q): \T{A}=A\}|$, we obtain
\begin{align*}
|\mathcal{P}_{2g}\backslash\mathcal{C}_{2g}| & = \frac{|\mathfrak{U}_{2g}|}{|\mathfrak{P}_{2g}|} = \frac{|\ker \lambda^{(2g)} |\cdot |\Sp_{4g}(\F_q)|}{|P_{2g}(\F_q)|} = |\ker \lambda^{(2g)} |\cdot \frac{q^{4g^2}\prod_{i=1}^{2g} (q^{2i}-1)}{q^{4g^2}\prod_{i=1}^{2g} (q^{i}-1)}\\
& = |\ker \lambda^{(2g)} |\prod_{i=1}^{2g} (q^{i}+1) .
\end{align*}
\end{proof}

The above proof shows that
\[
\left( D(E_{2g}),f\right)_g = N!\, q^{-Ng/2}|\ker(\lambda^{(g)})|^2\,\frac{|\mathfrak{U}_{g}|\cdot |\mathfrak{P}_{2g}|}{|\mathfrak{U}_{2g}|} \bar{f}.
\]
We believe that this is also the case for type $2_I^E$.

\begin{conj}\label{conj:type2I}
Let $f\in \C [y_v: v\in\F_2^g]_N$ be a cusp form of genus $g$ associated with type $2^E_{I}$. Then
\[
\left( D(E_{2g}),f\right)_g = N!\, 2^{2g-Ng/2}\frac{|\mathfrak{U}_{g}|\cdot |\mathfrak{P}_{2g}|}{|\mathfrak{U}_{2g}|} \bar{f},
\]
where
\[
\frac{|\mathfrak{U}_{g}|\cdot |\mathfrak{P}_{2g}|}{|\mathfrak{U}_{2g}|} =\frac{| O^+_{2g}(\F_2)|\cdot |P_{2g}\cap O^+_{4g}(\F_2)|}{ |O^+_{4g}(\F_2)|} 
= 2^{g^2-g}(2^g -1)\prod_{i=g}^{2g-1}(2^i+1)^{-1} .
\]
\end{conj}

Here we are using (see \cite{Wilson}) that 
$$|O^+_{2g}(\F_2)| = 2^{g^2-g+1} (2^g -1)\prod_{i=1}^{g-1}(4^i-1),$$
and
$$|P_{2g}\cap O^+_{4g}(\F_2)| = 2^{2g(2g-1)} \prod_{i=1}^{2g}(2^i-1).$$

\subsection{Computational evidence for the codes of type $2^E_{I}$}\label{sec:computation}

The results presented in sections \ref{sec:r<g} and \ref{sec:r=g} concern the codes of type $2^E_{II}$, $q^E$, $q_I^E$, but not of type $2^E_{I}$. This is because in this case the Clifford-Weil group is related to the orthogonal group $O^+(\F_2)$ for which we do not know the double coset decomposition used in the proof of Lemma \ref{lem:double-coset-decomp}. Most probably this is just a technical difficulty. In this section we would like to discuss computational results related to codes of type $2^E_{I}$ and length $N=16$ which were carried out with the help of SageMath. The calculations were performed for $g\in\{ 1,2\}$.

The starting point for our calculations is an alternative formula for the Eisenstein series:
\begin{equation}\label{eq:Siegel-Weil 2_II}
E_{2g}(\underline{X})=\prod_{1\leq i\leq \frac{N}{2}-1} (2^{2g}+2^i)^{-1}\sum_{C}\cwe_{2g}(C,\underline{X}) ,
\end{equation}
where the sum is over all self-dual codes $C$ of type $2_I^E$ and length $N$ (see \cite[Theorem 4.10]{NRS01}). Due to the identity \eqref{eq:restriction}, for $f\in\C[y_{v}: v\in \F^{g}]_N$ we have 
\[
\left( D(E_{2g}(\underline{X})),f\right)_g
= \prod_{1\leq i\leq \frac{N}{2}-1} (2^{2g}+2^i)^{-1}\sum_{C}\cwe_{g}(C,\underline{x})\left( \cwe_{g}(C,\underline{y}),f(\underline{y})\right)_g .
\]
Observe also that if a code $C'$ may be obtained from a code $C$ by permuting the entries of all the vectors, i.e. $C$ and $C'$ are permutation-equivalent, then $\cwe_{g}(C)=\cwe_{g}(C')$ for every genus $g$. Hence,
\begin{equation}\label{eq:inner product 1}
\left( D(E_{2g}),f\right)_g
= N!\prod_{1\leq i\leq \frac{N}{2}-1} (2^{2g}+2^i)^{-1}\sum_{C}\frac{1}{|\mathrm{Aut}(C)|}\cwe_{g}(C)\left( \cwe_{g}(C),f\right)_g ,
\end{equation}
where $\mathrm{Aut}(C)=\{\sigma\in\mathrm{S}_N:\sigma C=C\}$ is the automorphism group of the code $C$ and the sum is over all permutation-nonequivalent self-dual codes of type $2_I^E$ and length $N$.

The number of permutation-nonequivalent self-dual codes of type $2_I^E$ and length $N\leq 34$ is given in \cite[Table 12.1, column (f)]{NRS06}; there, in chapter 12.2 one can also read about their classification. The classification (of a larger group of codes) for $N\leq 20$ was carried out by Pless in \cite{Pless}. In \cite{Pless} one can also find the sizes of the automorphism groups and formulas for the weight enumerators of genus $1$. We collected them in Table \ref{table:N=16}. The first column of the table contains the name of a code (taken from \cite[Table 12.7]{NRS06}); the first three are of length $16$, and the next four are obtained by taking direct sums of codes of smaller length with the repetition code $i_2=\{ 00,11\}$. The last column lists the sizes of the automorphism groups, and the columns in the middle contain information on the coefficients of the genus-$1$ complete weight enumerators, for example:
\[
\cwe_1(E_{16})(x_0,x_1)=x_0^{16}+x_1^{16} +28\left( x_0^{12}x_1^4 + x_0^{4}x_1^{12}\right) + 198x_0^8x_1^8
\]
The suggestive names for the consecutive terms of the weight enumerators are taken from \cite{Runge96}. In case of genus $1$ the tuple $(a,b)$ denotes a sum $x_0^ax_1^b + x_0^bx_1^a$ (if $a\neq b$) or a monomial $x_0^ax_1^b$ (if $a= b$), and $(a)=(a,0)$.

\begin{table}[ht] 
\centering
\bgroup\def\arraystretch{1.2} 
\begin{tabular}{ccccccccc}
 $C$ & & (16) & (14,2)	& (12,4) & (10,6) & (8,8) &	& $\mid\!\mathrm{Aut}(C)\!\mid$ \\
\hline
$E_{16}$ & & 1 & 0	& 28 & 0 & 198 & & 5160960 \\
$F_{16}$ & & 1 & 0	& 12 & 64 & 102 & & 73728 \\
$A_8^2$ & & 1 & 0 & 28 & 0 & 198 & & 3612672 \\
$D_{14}\oplus i_2$ & & 1 & 1 & 14 & 63 & 98 & & 112896 \\
$B_{12}\oplus i_2^2$ & & 1 & 2 & 16 & 62 & 94 & & 184320 \\
$A_8\oplus i_2^4$ & & 1 & 4 & 20 & 60 & 86 & & 516096 \\
$i_2^8$ & & 1 & 8 & 28 & 56 & 70 & & 10321920 \\
\\
\end{tabular}\egroup
\caption{A list of all inequivalent codes of length $16$ of type $2^E_{I}$ together with the sizes of their automorphism groups and the coefficients of genus-$1$ complete weight enumerators.}\label{table:N=16}
\end{table}

It is known that the space of cusp forms of length $16$ of genus $1$ is $2$-dimensional. As basis elements we choose:
\[
f_1=\frac{1}{16}\left(\cwe_1(E_{16})-\cwe_1(F_{16})\right) = (12,4) -4 (10,6) +6 (8,8)
\]
\[
f_2=\frac18\left(\cwe_1(E_{16})-\cwe_1(i_2^8)\right) = -(14,2) -7(10,6) + 16(8,8)
\]
With this data the formula \eqref{eq:inner product 1} gives:
\[
\left( D(E_{2}),f_1\right)_1 = c f_1 ,\qquad \left( D(E_{2}),f_2\right)_1 = c f_2
\]
where
\[
c=16!\prod_{1\leq i\leq 7} (4+2^i)^{-1}\cdot 19388160 = \frac{16!}{2^6\cdot 3} .
\]
Note that the value of the constant $c$ agrees with the value predicted by Conjecture \ref{conj:type2I}, and we note here that $f_1$ and $f_2$ have real coefficients. As we show below, the same happens when we take $g=2$.

In order to find a cusp form $f$ of genus $2$ and compute $\left( D(E_{4}),f\right)_2$ we computed the coefficients of the genus-$2$ complete weight enumerators. It's clear from the definition that the first coefficients are precisely the same as in genus $1$; the remaining are listed in Table \ref{table:N=16,g=2}. That is, for example,
\begin{align*}
\cwe_2(E_{16})(x_0,x_1) = &\, (16) +28(12,4) + 198(8,8)\\
& + 420(8,4,4) +336(10,2,2,2) + 4704(6,6,2,2) +	29400(4,4,4,4) .
\end{align*}
As before, we omit zeros in tuples, and the tuple $(a_0,a_1,a_2,a_3)$ denotes the sum of all\footnote{This is no longer the case for genus $g\geq 3$, cf. \cite{Runge96}.} possible, pair-wise different monomials of the form $x_{v_0}^{a_0}x_{v_1}^{a_1}x_{v_2}^{a_2}x_{v_3}^{a_3}$ for distinct $v_0,v_1,v_2,v_3\in\{ 00,01,10,11\}$. In particular, now in genus $2$:
\[
(16) = x_{00}^{16} + x_{01}^{16} + x_{10}^{16} + x_{11}^{16} 
\]
is a different polynomial than in genus $1$.

{\footnotesize 
\begin{table}[ht] 
\centering
\bgroup\def\arraystretch{1.2}
\begin{tabular}{ccccccccccc}
 $C$ &  (12,2,2) & (10,4,2)	& (8,6,2) & (8,4,4)	& (6,6,4) & (8,4,2,2) &	(10,2,2,2) & 
 (6,6,2,2) & (6,2,4,4) & (4,4,4,4) \\
\hline
$E_{16}$ &  0 & 0 & 0 & 420 & 0 & 0 & 336 & 4704 &	0 &	29400 \\
$F_{16}$ &  0 & 0 & 0 & 84 & 192 & 576 & 48 & 1056 & 3264 & 8088 \\
$A_8^2$ &  0 & 0 & 0 & 420 & 0 & 0 & 336 & 4704 &	0 &	29400 \\
$D_{14}\oplus i_2$  & 0 & 14 &	49 & 98 & 196 & 672 & 84 & 1176 & 3038 & 7056 \\
$B_{12}\oplus i_2^2$  & 2 & 30 & 94 & 120 & 212 & 750 & 120 & 1264 & 2820 & 6120 \\
$A_8\oplus i_2^4$ &  12 & 68 & 172 & 188 & 280 & 852 & 192 & 1344 & 2408 & 4536 \\
$i_2^8$  & 56 & 168 & 280 & 420 & 560 & 840 & 336 & 1120 & 1680 & 2520 \\
\\
\end{tabular}\egroup
\caption{A list of all inequivalent codes of length $16$ of type $2^E_{I}$ together with the coefficients of genus-$2$ complete weight enumerators. The rest of the coefficients is given in Table \ref{table:N=16}.}\label{table:N=16,g=2}
\end{table}}

The subspace of cusp forms of genus $2$ is $1$-dimensional. As a generator we can choose
\begin{align*}
f & = (A_8\oplus i_2^4 + F_{16} - 2B_{12}\oplus i_2^2 )/8\\
& = (12,2,2)+ (10,4,2) -2(8,6,2) + 4(8,4,4) + 6(6,6,4) - 9(8,4,2,2)\\
&\hspace{0.5cm} - 16(6,6,2,2) + 4(6,2,4,4) + 48(4,4,4,4).
\end{align*}
Then
\[
\left( D(E_{4}),f\right)_2 = 16!\prod_{1\leq i\leq 7} (16+2^i)^{-1}\cdot 9953280\, f = \frac{16!}{2^{10}\cdot 3\cdot 5}\, f
\]
which agrees with the conjectural value and here we have $\bar{f}=f$.

\begin{rem}
We carried out analogous calculations for the codes of type $2_{II}^E$, length $N=24$ and $g\in\{ 1,2\}$. The resulting constant agrees with the constant obtained in the Theorem.
\end{rem}

\section{Application of the main theorem to the ``basis problem''}\label{sec:basis problem}

In this last section we present an application of the Theorem to the ``basis problem''. Namely, we write a formula which expresses any cusp form (in the sense of section \ref{sec:cusp forms}) as a linear combination of complete weight enumerators of codes (the analogues of ``theta series''). This formula is inspired by the one given by B\"ocherer in \cite[Satz 22]{B83}. In order to achieve this we need a Siegel-Weil type formula, which connects our Siegel-type Eisenstein series with the weight enumerators of codes. Such formulas are known for the types $2_I^E$ (see equation \eqref{eq:Siegel-Weil 2_II}), $2_{II}^E$ and $p_1^E$ (see Theorems 6.3 and 7.2 in \cite{NRS01}). In particular,

\begin{enumerate}
\item for type $2_{II}^E$:
$$E_g(\underline{x})=\prod_{0\leq i<\frac{N}{2}-1} (2^g+2^i)^{-1}\sum_{C}\cwe_g(C,\underline{x})$$
where the sum is over all doubly-even self-dual codes $C$ of type $2_{II}^E$ and length $N$;
\item for type $p_1^E$:
$$E_g(\underline{x})=\prod_{0\leq i<\frac{N}{2}-1}(p^g+p^i)^{-1}\sum_{C}\cwe_g(C,\underline{x})$$
where the sum is over all self-dual codes $C$ of type $p_1^E$ and length $N$.
\end{enumerate}

\begin{cor}
Let $f$ be a cusp form of genus $g$ and length $N$ associated with a type $2^E_{II}$ or $p_1^E$. Then:
\[
f = \frac{b}{cN!} \sum_C \left(f, \cwe_g(C)\right)_g \cwe_g(C),
\]
where the sum is over all codes $C$ of length $N$ and type $2^E_{II}$ or $p_1^E$, $c$ is the constant from the Theorem, and 
\[
b=\begin{cases}
\prod_{0\leq i<\frac{N}{2}-1} (2^{2g}+2^i)^{-1} , & \mbox{for type } 2^E_{II}\\
\prod_{0\leq i<\frac{N}{2}-1}(p^{2g}+p^i)^{-1}, & \mbox{for type } p^E_{1} .
\end{cases}
\]
\end{cor}
\begin{proof}
The Siegel-Weil type formulas for $E_{2g}$ may be abbreviated to the form
\[ E_{2g}(\underline{X})=b \sum_{C}\cwe_{2g}(C,\underline{X}),\]
where $b$ is as above. By the identity \eqref{eq:restriction},
\[ D(E_{2g})(\underline{x},\underline{y})=b \sum_{C}\cwe_g(C,\underline{x})\,\cwe_g(C,\underline{y}),\]
and thus for a cusp form $f(\underline{y})$ we have
\[
\left( D(E_g(\underline{x},\underline{y}), f(\underline{y})\right)_g = b \sum_C \left( \cwe_g(C,\underline{y}),f(\underline{y})\right)_g \cwe_g(C,\underline{x}).
\]
On the other hand, by the Theorem, $\left( D(E_g(\underline{x},\underline{y}), f(\underline{y})\right)_g =cN!\, \bar{f}(\underline{x})$. The formula then follows from the fact that the inner product is hermitian and the weight enumerators have real coefficients.
\end{proof}

\begin{rem}
If our Theorem can be extended to cover type $2_I^E$ (see discussion in section \ref{sec:computation}), then the above corollary holds also in this case with $b=\prod_{1\leq i\leq \frac{N}{2}-1} (2^{2g}+2^i)^{-1}$.
\end{rem}

\section*{Appendix: The doubling method in the theory of automorphic forms}  

This section is addressed to researchers in coding theory who wish to know some basic facts of the original doubling method, as it appears in the theory of automorphic forms, which was our motivation and inspiration for this paper. Of course, we will not present the method in its full generality but rather restrict ourselves to some very concrete cases (Siegel modular forms, holomorphic case, classical setting, etc). For more details we refer the reader to the book of Shimura \cite{Shimura_book} for a presentation in the classical language and to the original monograph of Piatetski-Shapiro and Rallis \cite{GP-SR} for the automorphic representation approach.\newline

We start with recalling some basic facts regarding Siegel modular forms. For more details we refer the reader to \cite{123}. For a fixed positive integer $g$, called degree or genus, the Siegel's upper half space of degree $g$ is
\[
\mathbb{H}_g = \{ Z \in M_g(\mathbb{C}): Z= \T{Z},\; \Im(Z) > 0 \} ,
\]
where by $\Im(Z) > 0$ we mean that the imaginary part of the complex matrix $Z$ is positive definite. On this space there is a transitive action of the symplectic group $\Sp_g(\mathbb{R})$ given by
\[
\gamma Z := (AZ+B)(CZ+D)^{-1},
\]
where we write $\gamma = \(\begin{smallmatrix}
    A & B \\ C & D 
    \end{smallmatrix}\) \in \Sp_g(\mathbb{R})$ with $A,B,C,D \in M_g(\mathbb{R})$. 

A Siegel modular form of  degree $g$, integer weight $k$ and level $\Gamma =\Gamma_g := \Sp_g(\mathbb{Z})$ is a holomorphic function $F$ on $\mathbb{H}_g$ which is invariant under the action of all $\gamma =\(\begin{smallmatrix}
    A & B \\ C & D 
    \end{smallmatrix}\)\in \Gamma$:
\[
F|_k\gamma (Z) := \det(CZ+D)^{-k} F(\gamma Z)= F(Z) .
\]
If $g=1$ one needs to impose an additional condition on the behavior of $F$ at infinity (the cusp of $\Gamma$). However if $g > 1$, then this condition is implied by holomorphy of $F$ (Koecher principle).

An example of a Siegel modular form of degree $g$ and even weight $k > g+1$ is an Eisenstein series of Siegel type
\[
E^{(g)}_k(Z) =\sum_{\gamma \in P\cap \Gamma \backslash \Gamma} 1|_k\gamma (Z) = \sum_{\(\begin{smallmatrix}
    A & B \\ C & D 
    \end{smallmatrix}\) \in P\cap \Gamma \backslash \Gamma} \det(CZ+D)^{-k}
\]
where $P \subset \Sp_{g}(\mathbb{R})$ is the Siegel parabolic subgroup consisting of matrices with lower left block (i.e. $C$) equal to the zero matrix, and $1$ is the constant function equal to $1$.

Another very interesting Siegel modular forms are the cuspidal ones. Given a Siegel modular form $F$ of degree $g$, one can obtain a Siegel modular form of degree $g-1$ by applying Siegel's $\Phi$-operator:
\[
(\Phi F)(Z) = \lim_{y \rightarrow \infty} F\left( \begin{pmatrix} Z & 0 \\ 0 & iy\end{pmatrix}\right) ,\qquad Z \in \mathbb{H}_{g-1} .
\]
If $\Phi(F) = 0$, then $F$ is called a cusp form. For two Siegel modular forms $F$ and $G$ of degree $g$ and weight $k$, one of which is a cusp form, we define
\[
<F,G> = \int_{\Gamma \setminus \mathbb{H}_g} F(Z) \overline{G(Z)} (\det\Im (Z))^{k} d^{*}Z
\]
where $d^{*}Z$ is a $\Gamma$-invariant volume element of $\mathbb{H}_g$.

We are now ready to present an idea of the doubling method from the classical point of view. Given two copies of $\mathbb{H}_g$ we define an embedding
\[
\Delta : \mathbb{H}_g \times \mathbb{H}_g \rightarrow \mathbb{H}_{2g},\qquad (Z_1,Z_2) \mapsto \begin{pmatrix} Z_1 & 0 \\ 0 & Z_2\end{pmatrix} .
\]
The embedding $\Delta$ is compatible with an embedding 
\[
\imath : \Sp_g(\mathbb{R}) \times \Sp_g(\mathbb{R}) \rightarrow \Sp_{2g}(\mathbb{R}),\qquad (\begin{pmatrix} A_1 & B_1 \\ C_1 & D_1 \end{pmatrix} , \begin{pmatrix} A_2 & B_2 \\ C_2 & D_2 \end{pmatrix} )\mapsto \begin{pmatrix}
    A_1 & 0 & B_1 & 0 \\
    0 & A_2 & 0 & B_2\\
    C_1 & 0 & D_1 & 0\\
    0 & C_2 & 0 & D_2
\end{pmatrix} 
\]
in the sense that $\imath(\gamma_1,\gamma_2) \Delta(Z_1,Z_2) = \Delta (\gamma_1 Z_1, \gamma Z_2)$.

Doubling method amounts to obtaining an identity of the form
\begin{equation}\label{eq:doubling identity}
<E^{(2g)}_k \left( \begin{pmatrix} Z_1 & 0 \\ 0 & Z_2
\end{pmatrix}\right) , F(Z_1)> = \lambda(F) \overline{F(-\overline{Z_2})},
\end{equation}
where $\lambda(F)$ is a constant, $E^{(2g)}_k$ is a Siegel-type Eisenstein series\footnote{In practice, this Eisenstein series depends also on a complex variable $s$.} of genus $2g$, weight $k$ and level $\Gamma$, and $F$ is a cusp form of genus $g$, weight $k$ and level $\Gamma$, which is an eigenfunction of the so-called Hecke operators. (It is known that a finite subset of such eigenfunctions comprises a basis of cuspidal Siegel modular forms of genus $g$ of given weight and level.) In order to do this one proceeds as follows:
\begin{enumerate}
\item[1.] Find a double coset decomposition $\Gamma_{2g} =\bigcup_{r} (P\cap \Gamma_{2g} )\,\tau_r\, \iota(\Gamma_g\times \Gamma_g)$.
\item[2.] Use step $1.$ to write $E^{(2g)}_k = \sum_r E_r$ and compute $<E_r \left( \begin{smallmatrix} Z_1 & 0 \\ 0 & Z_2 \end{smallmatrix}\right) , F(Z_1)> $ separately for each $r$. The cuspidality of $F$ should force all but one of these inner products to be zero. 
\item[3.] The assumption that $F$ is a Hecke eigenform should imply that what was obtained in step $2.$ leads to the doubling identity \eqref{eq:doubling identity}.
\end{enumerate}
As a result of step $3.$ we obtain a constant $ \lambda(F)$ which depends on the eigenvalues of $F$. This number is of great significance in the theory of Siegel modular forms since it is related to a value of the standard $L$-function attached to $F$. The identity \eqref{eq:doubling identity} and its generalizations have been also used to derive analytic and algebraic properties of the aforementioned $L$-function. \newline

\noindent \textbf{Acknowledgments:} The first named author would like to thank R. Gaffney and R. Psyroukis for helpful discussions and the Faculty of Mathematics and Computer Science of Adam Mickiewicz University for hospitality in September 2024. The second named author would like to thank W. Gajda for helpful discussions on matrix groups.

\end{document}